\documentclass[11pt,a4paper]{article}
\usepackage[blocks]{authblk}
\usepackage{graphicx} 
\usepackage{amsmath}
\usepackage{hyperref}
\usepackage{float}
\usepackage{caption}
\usepackage{amssymb}
\usepackage{comment}
\usepackage{amsthm}
\usepackage{bbold}
\usepackage{xcolor}
\usepackage{cases}

\usepackage{stmaryrd}
\usepackage{algpseudocode}
\usepackage{algorithm}
\usepackage[numbers]{natbib}

\usepackage{geometry}
\graphicspath{{Images/}}

\newtheorem{thm}{Theorem}[section]
\newtheorem{prop}[thm]{Proposition}

\numberwithin{equation}{section}

\theoremstyle{definition}

\newtheorem{assp}{Assumption}

\newtheorem{appxlem}{Lemma}[section]

\title{Hematopoiesis as a continuum: from stochastic compartmental model to hydrodynamic limit}

\author[1]{Vincent Bansaye
\thanks{Email: \texttt{vincent.bansaye@polytechnique.edu}}
}
\author[1]{Ana Fern\'andez Baranda
\thanks{Email: \texttt{ana.fernandez-baranda@polytechnique.edu}}
}
\author[2]{St\'ephane Giraudier
\thanks{Email: \texttt{stephane.giraudier@aphp.fr}}
}
\author[1]{Sylvie M\'el\'eard
\thanks{Email: \texttt{sylvie.meleard@polytechnique.edu}}
}
\affil[1]{{\small CMAP, Ecole Polytechnique, CNRS, Institut
    polytechnique de Paris, Inria, Palaiseau, France.}}
\affil[2]{{\small Universit\'e Paris Cit\'e, H\^opital Saint-Louis, INSERM U1131, Paris, France.}}

\begin{document}

\maketitle

\begin{abstract}
We consider a multiscale stochastic compartmental model with three types of cells (stem cells, immature cells and mature cells) which combines cell proliferation and cell differentiation. We derive a hydrodynamic limit when the number of immature compartments goes to infinity obtaining a partial differential equations system with boundary conditions, modelling hematopoiesis as a continuum. We assume that proliferation and differentiation are regulated and let the corresponding rates depend on the number of mature cells. This leads us to model the dynamics of the population by a Markov process in continuous time and discrete space, which does not satisfy the branching property. We prove the convergence in law of the stem and mature cells population size processes and of the empirical measures of the immature cells dynamics, conveniently rescaled, to the unique triplet involving coupled functions and a measure, which are solutions of a deterministic measure valued equation with boundary dynamics. The cell differentiation induces a transport term in space and the main difficulty comes from the boundary effects coming from stem and mature cells. We also prove that the limiting measure admits at each time a density with respect to Lebesgue measure and can be characterized as solution of a partial differential equation.
\end{abstract} \ \\

\noindent\textbf{Corresponding author}: Ana Fern\'andez Baranda, ana.fernandez-baranda@polytechnique.edu \\


\newpage

\section{Introduction}
Hematopoiesis is the process by which blood cells are produced. This process starts with hematopoietic stem cells (denoted HSCs) that go through differentiation: the cells gradually become more and more specialized until they finally become mature blood cells with a specific function in the body. These immature cells can also self-renew through division, where two daughter cells are created from a mother cell; or die. The most mature myeloid cells (myeloid leucocytes, erythrocytes, platelets) do not divide and can only be created by differentiation or die with different half-life durations. As shown by \citet{Smith}, the growth and differentiation of hematopoietic immature cells are regulated by signals from the microenvironment, mainly through cytokines secreted by mature cells. 
The process of hematopoiesis has long been described as a hierarchical step-wise process. Biological observations of hematopoietic cells showed phenotype differences among the progenitor cells which allowed to create distinct groups of cells by phenotype as explained in \citet{Cheng}. In this framework, HSCs differentiate a fixed number of times, becoming at each step a different more specialized progenitor, until finally becoming mature blood cells. Regardless, this classical modeling oversimplifies the complexity that hematopoietic cells have. Indeed, by using new techniques of observation at single-cell level (single-cell RNA-Seq analysis), this hierarchical discrete structure can be questioned as presented in \citet{Quesenberry} and \cite{Cheng}. Current hypothesis concerning hematopoiesis development defined types of hematopoietic maturation stages are able to divide into more subtypes, leading to a continuum of hematopoiesis differentiation rather than a discrete model (see \citet{Wilson} and \citet{Pietras}). The authors in \citet{Zhang} and \cite{Quesenberry} suggest that hematopoiesis is a continuum process where the phenotype of stem cells widely fluctuates over time, giving a heterogeneity among cells rather than clearly defined homogeneous phenotypes that allow to separate into different cell types. There are no obvious boundaries between the different hierarchical levels of hematopoietic cells. Moreover, \cite{Quesenberry} explained that the phenotype of a cell changed during the cell cycle which would not allow for a clear separation of cells into distinct categories. All of these arguments suggest that rather than being a step-wise differentiation, it is a progressive process. Then, rather than having different maturation cell types, we would be talking about maturation level, where the most immature cells are HSC and the final stage is mature cells. 
The emergence of HSC that represent the upper compartment of the hematopoietic tree is also originating during embryogenesis from the AGM region where very few cells (mainly 1 to 3) (see, for example, \citet{Kumaravelu}, \citet{Rybtsov} and \citet{Tober}) emerge, migrate to the liver where the cells expand and give rise to HSC. Regardless, the precise mechanisms, number of cells at each step, and death rates at each step of this process are not well known. Then, we firstly hypothesized a infinite number of steps from HSC to mature cells i.e. a continuum in differentiation process.\\

Previous modeling of hematopoiesis has been widely based on finite compartmental models where each compartment represents a distinct type of cell corresponding to the different maturation stages. This compartmental approach can be seen, for example, in \citet{Bonnet}, \citet{Dingli}, \citet{Knauer}, \citet{MonPere} and \citet{Stiehl}. In all of these models, differentiation is directly linked to division. The hypothesis of regulation is often included in modeling as in \cite{Knauer}, \cite{MonPere} and \cite{Bonnet} through a feedback loop depending on the number of mature cells. While all of these models consider at least three compartments corresponding to HSCs, immature cells and mature cells, the number of immature compartments varies from one to the other depending on the phenomenon to study. Indeed, while in \cite{Knauer} they consider that a single immature cell type is enough to properly explain oscillatory dynamics, in \cite{Bonnet} they prove that 6 compartments at least are necessary to explain the amplification in the erythroid lineage and in \cite{Dingli} they suggest that up to 31 types are needed for amplification. \\

In this article, we derive a deterministic model of an hematopoiesis continuum from a stochastic compartmental model. Starting from a continuous in time pure jump Markov process with self-renewal and cell differentiation, we study the limiting behavior when the number of immature compartments goes to infinity. In this scenario, division and differentiation are separate phenomena and follow slow-fast dynamics where division and death rates are much smaller than differentiation rates. Indeed, it is well known that the number of cell divisions from stem to mature cells is finite, whereas differentiation becoming a continuous process, the biological changes in the cell's state become smaller and smaller and happen faster and faster. The renewal and differentiation rates both depend on the number of mature cells (as motivated previously), making the model nonlinear. The stochastic system depends on a scaling parameter $N$ which scales both the (large) order of magnitude of the stem cells population size as well as the number of compartments. We consider separately, on one hand the stem and mature cells dynamics $( X_1^N(t),\, t\ge 0)$ and $( X_N^N(t),\, t\ge 0)$ which are involved in division (for the first ones) and deaths  (for the last ones), and on the other hand the immature compartments dynamics in large number $( X_i^N(t), \, t\ge 0\,;\, 2\le i\le N-1)$, which depend on the fast differentiation rates. The latter are gathered in an empirical measure-valued process on spatial sites corresponding to the compartment levels, defined by \begin{gather*}
    \mu_t^{N} = \, \frac{1}{N} \, \sum_{i = 2}^{N - 1} X_i^{N} (t) \, \delta_{i/N}.
\end{gather*} This results in three coupled processes $\left( X_1^N(t), \, \mu_t^N, \, X_N^N(t) \right)_{t \geq 0}$, one describing the stem cells, another for the immature cells, and a third for the mature cells. When doing so, we find a boundary condition issue: the dependence of the mature cells dynamics on the final immature compartment $X^N_{N-1}$. We overpass this difficulty by finding an equation only depending on the processes of the triplet $\left( X_1^N(t), \, \mu_t^N, \, X_N^N(t) \right)_{t \geq 0}$. Our main result is Theorem \ref{thm:conv}, stating that these dynamics converge, as $N$ tends to infinity, to a deterministic nonlinear system $((a(t), \mu_t, z(t)), \, t\ge 0)$ composed of two ordinary differential equations satisfied by $a$ and $z$ describing limiting stem and mature dynamics, and a measure-valued equation satisfied by $\mu$ describing the immature dynamics for the continuum of compartments. We obtain this convergence through a tightness-identification-uniqueness procedure. We first use the semimartingale decomposition of the processes to prove tightness results. We characterize the limiting values as solutions of the limiting system. Differentiation behaves now as a transport term in the measure-valued equation. The uniqueness in the space of measure-valued functions was a difficult point. To prove it, we take into account the associated mild equation obtained by using the flow of the ordinary differential equation representing the cell maturation dynamics. This mild formulation is also used to prove that when the initial empirical measure admits a density, then this property is preserved over time. \\

Similar limiting behaviors have been widely studied for interacting particle systems over $\mathbb{Z}^d$ when both the number of particles and the number of sites go to infinity and the number of particles is constant, see for example \citet{Kipnis} and \citet{DeMasi}.
Let us also mention \citet{Blount} and \citet{Perrut} which include birth and deaths in each site, for respectively one and two species with competition, and whose hydrodynamic limit is a reaction diffusion process. 
Up to our knowledge, much less work has been achieved for hydrodynamic limits combining birth and death processes, transport and boundary effects. We refer to \citet{Bahadoran} for a model where the birth and death appear in reservoirs located at the boundary. Let us also observe that when the interactions are neglected, our process with $N$ compartments is a $N$ types branching process, as studied in \citet{Athreya}, \citet{Mode}, \citet{Kimmel}. We are interested here in the approximation of the process for a large number of types $(N\rightarrow \infty)$. Taking into account the regulation, the branching property is broken and the techniques we use involve tightness of the sequence of processes and identification of the limit rather than the branching structure of the process. \\

\citet{Doumic} studied the continuous limiting values of a deterministic compartmental model for cellular differentiation, with a feedback loop in the differentiation. Their work provided the framework for limiting continuous models for cellular differentiation. In contrast, we present a stochastic compartmental model that describes the dynamics of the population of cells based on the individual cell behavior and consider regulation in both the differentiation and regulation rates. Besides, we extend the space of solutions for the limiting system to the set of finite non-negative measures on $[0, 1]$ that do not necessarily have a density. The relation between our work and \cite{Doumic} is done in the particular case where densities exist at each time for the limiting measures $\mu_t$. Our result proves the uniqueness of the limiting values obtained in \cite{Doumic}. \\ 

The structure of the paper is as follows. In Section \ref{sec:disc}, we construct the discrete jump process for any fixed number $N$ of compartments including the feedback mechanism on the differentiation and division rates. Section \ref{sec:cont} is devoted to the proof of the convergence result. In Section \ref{sec:den}, we show that if the initial condition of the measure gathering the immature compartments admits a density, then, this property propagates through time, allowing us to find an associated PDE system with two boundary ODEs. In Section \ref{sec:sim}, we illustrate our results with simulations for both the stochastic and deterministic models. 

\bigskip
\textbf{Notation}: $\mathcal{C}(S)$ and $\mathcal{C}^1(S)$ denote respectively the sets of continuous and continuously differentiable functions on the set $S$. 

The set of finite non-negative measures on $[0,1]$ endowed with the weak topology is denoted by $(M_F,w)$. Note that this topology is metrizable with the following metric. Let $BL(S)$ be the set of bounded Lipschitz continuous functions on the set $S$ with the norm 
\begin{equation*}
   \forall g \in BL(S), \quad || g || = \sup_{x \in S} |g(x)| + \sup_{x, y \in S, \, x \neq y} \frac{|g(x) - g(y)|}{|x - y|} ,
\end{equation*} 
and define the unit ball
\begin{equation*}
    BL_1(S) = \{ g \in BL(S), \, || g || \leq 1 \}.
\end{equation*}
For two measures $\nu^1, \, \nu^2 \in \mathcal{M}_F$, we define 
\begin{equation*}
    ||\nu^1 - \nu^2 ||_{BL} = \sup_{g \in BL_1([0, 1])} | \left< \nu^1 - \nu^2, g \right> |.
\end{equation*}
    
$C$ and $C_T$ will denote respectively a positive constant that can change from line to line and a positive constant depending on $T$ that can change from line to line. Both of these constants can depend on the parameters of the model, but not on the scaling parameter $N$ nor on any chosen test function.

\section{Discrete model for cellular differentiation} \label{sec:disc}
In this model, we characterize a cell by a quantitative trait $x \in [0, 1]$ representing its maturation level. A maturation level of $x = 0$ corresponds to HSC, one of $x = 1$ to mature cells and for $x \in (0, 1)$ we are talking about immature cells. We consider initially that there are $N$ different cell types from HSCs to mature cells. A cell of type $i \in \{1, ..., N \}$ will have a maturation level of $x = i/N$. The parameter $N$ also defines the speed of differentiation in immature compartments. Indeed, the more maturation stages the cell has to pass by before becoming a mature cell, the faster it has to go through them in order to maintain the time scale of mature cells creation of order one. We define a discrete jump process $(X(t), \, t \geq 0) = ((X_1(t), ..., X_N(t)))_{t \geq 0}$ where $X_i(t)$ is the number of cells of type $i$ at time $t$. The goal is to study the limiting behavior of this system when $N \rightarrow \infty$. \\ \, \\
We introduce the functions $(x,z) \to r(x,z)$, $(x,z) \to m(x,z)$ defined on $[0,1]\times [0, \infty)$ that correspond to the division and differentiation rates for a cell of maturation level $x$ when the mature population is of size $z$. The dependence on the mature population size gathers the regulation factors on the division and differentiation rates. In the whole paper, we work with the following assumption.
\begin{assp} \label{assp:rm}
    We assume that the functions $r$ and $m$ are continuous and non-negative functions, that $m$ belongs to $\mathcal{C}^1([0, 1] \times [0, \infty))$, and that there exist constants $\widehat{r}, \, \widehat{m}, \, m_{min} \in \mathbb{R}^*_+$ such that for all $(x, z) \in [0, 1] \times[0, \infty)$,
    \begin{equation*}
        0 \leq r(x, z) \leq \widehat{r}, \quad m_{min} \leq m(x, z) \leq \widehat{m}.
    \end{equation*}
    Furthermore, we assume that there exist $L_m > 0$ and $L_r > 0$ such that for any $x_1, \, x_2 \in [0,1]$ and any $z_1, \, z_2 \in [0, \infty)$,
    \begin{gather*}
        |m(x_1, z_1) - m(x_2, z_2)| \leq L_m \, \left(|x_1 - x_2| + |z_1 - z_2| \right), \\
        \left| \frac{\partial}{\partial x} m(x_1, z_1) - \frac{\partial}{\partial x}m(x_2, z_2) \right| \leq L_m \, \left(|x_1 - x_2| + |z_1 - z_2| \right), \\
        \quad |r(x_1, x_1) - r(x_2, z_2)| \leq L_r \, \left(|x_1 - x_2| + |z_1 - z_2| \right).
    \end{gather*}
\end{assp} \, \\
The boundedness by above of the division and differentiation rates will ensure, in particular, the control of the first moment and the non-explosion of the stochastic process, while the Lipschitz continuity conditions on $m$, $\partial m/ \partial x$ and $r$ will guarantee uniqueness of the deterministic limit objects. In order to ensure creation of mature cells, the differentiation rate function has to be bounded by below by some $m_{min} > 0$. Indeed, if this was not the case and $m(x, z) = 0$ for a fixed $z \in [0, \infty)$ and some $x \in [0, 1]$, then the cells may not be able to attain a maturation level greater than $x$. \\ \, \\
For the $N$ subpopulations of cells we have the following transitions:
\begin{itemize}
    \item Each cell of type $i$, $(1 \leq i \leq N - 1)$ divides at rate $r(i/N, X_N / N)$ creating two cells of same type,
    \item Each cell of type $1$ becomes a cell of type $2$ at rate $m(1/N, X_N/N)$,
    \item Each cell of type $i$, $(2 \leq i \leq N - 1)$ becomes a cell of type $i + 1$ at rate $Nm(i/N, X_N/N)$,
    \item Each cell of type $N$ dies at rate $d$,
\end{itemize}
where $d$ is a positive constant. For simplicity, we did not consider deaths in immature compartments and the parameter $d$ is not regulated by the mature size population. We believe that these dynamics could be included by directly adapting the proofs presented here. \\ \, \\
Since the differentiation rate of the immature compartments is accelerated by the number of compartments, each of these immature compartment both fills and empties very fast. Then, the population size of stem and mature cells are of same order of magnitude. This assumption is not contradictory with the amplification phenomenon observed between HSCs and mature cells compartments (see simulations in Figure \ref{fig:disc}). We focus on the case where the population sizes of stem and mature cells are of order $N$. Hence, the $N - 2$ immature compartments are expected to have an order of magnitude of $1$. Thus, we consider the scaling
\begin{equation*}
    X_1^{N}(t) = \frac{X_1(t)}{N}, \quad X_i^{N}(t) = X_i(t), \quad X_N^{N}(t) = \frac{X_N(t)}{N}.
\end{equation*}\ \\
We are thus working with a multitype density dependent pure jump Markov process $(X^N(t),\, t\ge 0)$, whose trajectorial representation is given   using Poisson point measures. Let $( \mathcal{N}_i^j: 1 \leq i \leq N, \, j \in \{r, m, d \})$ be independent Poisson point measures with intensity $ds \, d\theta$ on $\mathbb{R}_+^2$. We introduce $(\mathcal{F}_t)_t$ as the canonical filtration generated by the initial condition $X^N(0)$ and these Poisson point measures. Then, the process $X^N(t)$ is the unique strong solution in $\mathbb{D} \left( [0, T], \, {\frac{\mathbb{N}}{N}}\times \mathbb{N}^{N-2} \times {\frac{\mathbb{N}}{N}} \right)$, for all $T > 0$, of the following system of stochastic differential equations
\begin{align} \label{eq:stoch}
    X_1^N(t) =& \, X_1^N(0) + \frac{1}{N} \, \int_0^t \int_{\mathbb{R}^+} \mathbb{1}_{\theta \leq N \, r(1/N, X_N^N(s^-)) \, X_1^N(s^-)} \, \mathcal{N}_1^r(ds, \, d\theta)\nonumber \\&- \frac{1}{N} \, \int_0^t \int_{\mathbb{R}^+} \mathbb{1}_{\theta \leq N \, m(1/N, X_N^N(s^-)) \, X_1^N(s^-)} \, \mathcal{N}_1^m(ds, \, d\theta) \ ; \\ \label{eq:stoch2}
    X_i^N(t) =& \, X_i^N(0) + \int_0^t \int_{\mathbb{R}^+} \mathbb{1}_{\theta \leq r(i/N, X_N^N(s^-)) \, X_i^N(s^-)} \, \mathcal{N}_i^r(ds, \, d\theta) \nonumber\\ &+ \int_0^t \int_{\mathbb{R}^+} \mathbb{1}_{\theta \leq N \, m((i - 1)/N, X_N^N(s^-)) \, X_{i - 1}^N(s^-)} \, \mathcal{N}_{i - 1}^m(ds, \, d\theta)\nonumber \\&- \int_0^t\int_{\mathbb{R}^+} \mathbb{1}_{\theta \leq N \, m(i/N, X_N^N(s^-)) \, X_i^N(s^-)} \, \mathcal{N}_i^m(ds, \, d\theta), \quad 2 \leq i \leq N - 1 \ ; \\\label{eq:stoch3}
    X_N^N(t) =& \, X_N^N(0) + \frac{1}{N} \, \int_0^t \int_{\mathbb{R}^+} \mathbb{1}_{\theta \leq N \, m((N - 1)/N, X_N^N(s^-)) \, X_{N - 1}^N (s^-)} \, \mathcal{N}_{N - 1}^m(ds, \, d\theta) \nonumber \\&- \frac{1}{N} \, \int_0^t \int_{\mathbb{R}^+} \mathbb{1}_{\theta \leq N \, d \, X_N^N(s^-)} \, \mathcal{N}_N^d(ds, \, d\theta).
\end{align}
By considering the compensated martingale measures of $(\mathcal{N}_i^j, 1 \leq i \leq N\,, \,j \in \{ r, m, d\})$, we can write the semimartingale decomposition of these equations for $ 2 \leq i \leq N - 1 $ as

\begin{subnumcases}{} \label{eq:disc}
    X^N_1(t) = & \, $X^N_1(0) + \int_0^t \left[ r\left( \frac{1}{N}, \, X_N^N(s) \right) - m\left( \frac{1}{N}, \, X_N^N(s) \right) \right] X^N_1(s) \, ds + M_1^N(t)\ $ ; \\
    \,\nonumber \\\label{eq:discb}
    X_i^N(t) = & $\, X_i^N(0) + \int_0^t N \, m\left( \frac{ i - 1}{N}, \, X_N^N(s) \right) X^N_{i - 1}(s) \, ds$ \nonumber \\
    \ & $+ \int_0^t \left[ r\left( \frac{i}{N}, \, X_N^N(u) \right) - N \, m\left( \frac{i}{N}, \, X_N^N(s) \right) \right] X_i^N(s) \, ds + M_i^N(t)$ ;\\
    \, \nonumber \\ \label{eq:discc}
    X_N^N(t) =& $\, X_N^N(0) + \int_0^t m\left( \frac{N - 1}{N}, \, X_N^N(s) \right) X^N_{N - 1}(s) \, ds - d \, \int_0^t X_N^N(s) \, ds + M_N^N(t)$,
\end{subnumcases}

where the processes $\,(M_j^N(t), 1 \leq j \leq N)_{t \geq 0}\,$ are square-integrable martingales whose quadratic variation can be easily computed (see Lemma \ref{lem:mart} in Appendix).

\medskip
In order to study the asymptotic behavior of the immature compartments of type $2 \leq i \leq N - 1$ which are in a different scaling, we consider the stochastic measure $\mu_t^N$ defined for any $t \geq 0$ as 
\begin{gather}\label{empmes}
    \mu_t^{N} = \, \frac{1}{N} \, \sum_{i = 2}^{N - 1} X_i^{N} (t) \, \delta_{i/N}.
\end{gather}
In all what follows, we will work with the triplet $\left( X_1^N(t), \, \mu_t^N, \, X_N^N(t) \right)_{t \geq 0}$ composed of the empirical measure containing all the immature compartments, and the boundaries that correspond to the cell and mature cells population sizes. From System (\ref{eq:disc}), we can write the semimartingale decomposition. The equations are written in such a way that they only depend on elements on the triplet and not on any $X_i^N$ for $2 \leq i \leq N - 1$.
\begin{prop} \label{prop:sde}
    Under Assumption \ref{assp:rm}, for any function $f \in \mathcal{C}([0, 1])$, for any $t \geq 0$, we have
    \begin{subequations} \label{eq:cont}
        \begin{alignat}{3}
            X_1^{N}(t) &= \, X_1^{N}(0) + A^N_1 + M_1^{N}(t) \\
            \left< \mu_t^N, f \right> &= \, \left< \mu_0^N, f \right> + A^{N, f} (t) + M^{N, f}(t) \\
            X_N^{N}(t) &= \, X_N^{N}(0) + A^N_{N}(t) + M^{N}(t)\, ,
        \end{alignat}
    \end{subequations}
    where 
    \begin{subequations} \label{eq:A}
        \begin{alignat}{5}
            A^N_{1}(t) =& \int_0^t \left[ r\left( \frac{1}{N}, \, X_N^N(s) \right) - m\left( \frac{1}{N}, \, X_N^N(s) \right) \right] X_1^{N}(s) \, ds \, ;\\
            A^{N, f}(t) =& \, f \left(\frac{2}{N}\right) \int_0^t m\left( \frac{1}{N}, \, X_N^N(s) \right) X_1^N(s) \, ds \nonumber \\
            &+ \int_0^t \left< \mu_s^N, f \, r( \cdot, \, X_N^N(s)) \right> \, ds + \int_0^t \left< \mu_s^N, \left(\Delta_{1/N} f\right) \, m( \cdot, \, X_N^N(s)) \right> \, ds \\
            & - f \left( 1 \right) \left[ \left< \mu_0^N, 1 \right> - \left< \mu_t^N, 1 \right> + \int_0^t \left< \mu_s^N, r(\cdot, \, X_N^N(s)) \right> \, ds + \int_0^t m \left( \frac{1}{N}, \, X_N^N(s) \right) X_1^N(s) \, ds \right] \, ; \nonumber\\
            A^N_N(t) =& \left< \mu_0^N, 1 \right> - \left< \mu_t^N, 1 \right> + \int_0^t \left< \mu_s^N, r( \cdot, \, X_N^N(s)) \right> \, ds \nonumber \\
            &+ \int_0^t m\left( \frac{1}{N}, \, X_N^N(s) \right) X_1^N(s) \, ds - d \, \int_0^t X_N^{N}(s) \, ds \,,
        \end{alignat}
    \end{subequations}
    $\Delta_{1/N} f$ corresponds to the discrete derivative of $f$ defined by
    \begin{equation*}
        \Delta_{h} f(x) = \, \frac{f(x + h) - f(x)}{h},
    \end{equation*}
    and the square-integrable martingales $M^{N,f}$, $M^N$ are given by
    \begin{align*}
         M^{N,f} &= \frac{1}{N} \sum_{i = 2}^{N - 1} \left( f \left( \frac{i}{N} \right) - f(1) \right) M_i^N \\
         M^{N} &= \frac{1}{N} \sum_{i = 2}^{N - 1} M_i^N + M_N^N.
    \end{align*}
\end{prop}

\begin{proof}
    The equation for $X_1^N(t)$ comes directly from (\ref{eq:disc}). Let us consider now
    $$ \left< \mu_t^N, f \right> = \, \frac{1}{N}\sum_{i = 2}^{N - 1} X_i^N(t) \, f \left( \frac{i}{N} \right)$$
     and use (\ref{eq:discb}). We deduce 
    \begin{align}
       \left< \mu_t^N, f \right> =& \, \frac{1}{N} \sum_{i = 2}^{N - 1} f \left( \frac{i}{N} \right) X_i^{N}(0) + \int_0^t \frac{1}{N} \sum_{i = 2}^{N - 1} f \left( \frac{i}{N} \right) r\left( \frac{i}{N}, \, X_N^N(s) \right) X_i^{N}(s) \, ds \nonumber\\
        &+ f \left( \frac{2}{N} \right) \int_0^t m\left( \frac{1}{N}, \, X_N^N(s) \right)  X_{1}^{N}(s) \, ds \nonumber \\
        &+ \int_0^t \frac{1}{N} \sum_{i = 2}^{N - 1} \left( f \left( \frac{i + 1}{N} \right) - f \left( \frac{i}{N} \right) \right) N \, m\left( \frac{i}{N}, \, X_N^N(s) \right) X_{i}^{N}(s) \, ds \nonumber\\
        &- f \left( 1 \right) \int_0^t m\left( \frac{N - 1}{N}, \, X_N^N(s) \right) X_{N - 1}^{N}(s) \, ds  + M^{N, f}(t) \nonumber \\
        =& \, \left< \mu_0^N, f \right> + \int_0^t \left< \mu_s^N, f \, r \left(\cdot, \, X_N^N(s) \right) \right> \, ds + f \left( \frac{2}{N} \right) \int_0^t  m\left( \frac{1}{N}, \, X_N^N(s) \right) X_{1}^{N}(s) \, ds \nonumber \\
        &+ \int_0^t \left< \mu_s^N, \Delta_{1/N}f \, m\left( \cdot, \, X_N^N(s) \right)  \right> ds \nonumber \\
        &- f \left( 1 \right) \int_0^t m\left( \frac{N - 1}{N}, \, X_N^N(s) \right) X_{N - 1}^{N}(s) \, ds +\frac{1}{N} \sum_{i = 2}^{N - 1} f \left( \frac{i}{N} \right) M_i^N(t). \label{ThebigEq}
    \end{align} \\ \, \\
    As we want the equations to only depend on elements of the triplet $(X^N_1, \mu^N, X^N_n)$, we need to get rid of the term involving $X_{N - 1}^N(t)$ in $\left< \mu_t^N, f \right>$. Note that for $2 \leq i \leq N- 1$, from (\ref{eq:discb}), we have that
    \begin{align} \label{eq:mXi}
        N \int_0^t m\left( \frac{i}{N}, \, X_N^N(s) \right) X_i^N(s) \, ds =& \, X_i^N(0) - X_i^N(t) + N \int_0^t m\left( \frac{i - 1}{N}, \, X_N^N(s) \right) X_{i - 1}^N(s) \, ds \nonumber\\
        &+ \int_0^t r\left( \frac{i}{N}, \, X_N^N(s) \right) X_i^N(s) \, ds + M_i^N(t).
    \end{align}
    We apply this formula for $2 \leq i \leq N -1$, we sum the corresponding equations and then divide both sides by $N$ to finally obtain that
    \begin{align} \label{eq:XN-1}
        \int_0^t m\left( \frac{N - 1}{N}, \, X_N^N(s) \right) X_{N -1}^N(s) \, ds =& \, \left< \mu_0^N, 1 \right> - \left< \mu_t^N, 1 \right> + \int_0^t m\left( \frac{1}{N}, \, X_N^N(s) \right) X_1^N(s) \, ds \nonumber \\
        &+ \int_0^t \left< \mu_s^N, r \left( \cdot, \, X_N^N(s) \right) \right> \, ds + \frac{1}{N} \sum_{i = 2}^{N - 1} M_i^N(t) .
    \end{align}
 Plugging \eqref{eq:XN-1} in \eqref{ThebigEq} yields (\ref{eq:cont}b) and (\ref{eq:A}b).
    Finally, for $X_N^N(t)$ we replace the term involving $X_{N - 1}^N(t)$ in (\ref{eq:discc}) using (\ref{eq:XN-1}).
\end{proof} \, \\
The quadratic variations of the martingales in System (\ref{eq:cont}) are given below.
\begin{prop} \label{prop:mart}
    The martingales $M^N_1$, $M^{N,f}$ and $M^N$ in Proposition \ref{prop:sde} are square-integrable and their quadratic variations are
    \begin{align*}
        \left< M_1^N \right>_t =& \, \frac{1}{N} \int_0^t \left( r\left( \frac{1}{N}, \, X_N^N(s) \right) + m\left( \frac{1}{N}, \, X_N^N(s) \right) \right)  X_1^{N}(s) \, ds\, ;\\
        \left< M^{N,f} \right>_t =& \, \frac{1}{N} \left[ \int_0^t \left< \mu_s^N, (f - f(1))^2 \, r\left( \cdot, \, X_N^N(s) \right) \right> \, ds + \frac{1}{N} \int_0^t \left< \mu_s^N, \left( \Delta_{1/N} f \right)^{2} \, m\left( \cdot, \, X_N^N(s) \right) \right> \, ds \right. \\
        & \, \, \, \, \, \, \, \left. + \left(  f \left(\frac{2}{N}\right) - f(1) \right)^2 \int_{0}^{t} m\left( \frac{1}{N}, \, X_N^N(s) \right) X_1^N(s) \, ds \right] ; \\
        \left< M^N \right>_t =& \, \frac{1}{N} \left[ \int_0^t \left< \mu_s^N, r\left( \cdot, \, X_N^N(s) \right) \right> \, ds + \int_0^t m\left( \frac{1}{N}, \, X_N^N(s) \right) X_{1}^N (s) \, ds + d \int_0^t X_N^{N}(s) \, ds \right] \,;\\
        \left< M_1^N, M^{N,f} \right>_t =&  \frac{1}{N} \left(  f(1)- f \left( \frac{2}{N} \right) \right) \int_0^t m \left( \frac{1}{N}, \, X_N^N(s) \right) X_1^N(s) \, ds \,;\\
        \left< M^{N,f}, M^N \right> =& \frac{1}{N} \left[ \int_0^t \left< \mu_s^N, \left(f(1)- f \left( \frac{i}{N} \right)  \right) \, r \left( \cdot, \, X_N^N(s) \right) \right> \, ds \, \right. \\
        & \ \ \ \ \ \ \ \ \left. + \int_0^t \left(f(1)- f \left( \frac{2}{N} \right)  \right) \, m \left( \frac{1}{N}, \, X_N^N(s) \right) X_1^N(s)\right] \, ds \,;\\
        \left< M_1^N, M^N \right>_t =& - \frac{1}{N} \int_0^t m \left( \frac{1}{N}, \, X_N^N(s) \right) X_1^N(s) \, ds\,,
    \end{align*}
     for any $t \geq 0$ and any $f \in \mathcal{C}([0, 1])$.
\end{prop} \, \\
\begin{proof}
    Recall that the compensated martingale measures $(\mathcal{N}_i^j, 1 \leq i \leq N\,, \,j \in \{ r, m, d\})$ are all independent. The quadratic variation of $M_1^N$ is computed in the proof of Lemma \ref{lem:mart}. The martingale $M^{N,f}$ is given by
    \begin{align*}
        M^{N,f}(t) =& \frac{1}{N} \sum_{i = 2}^{N - 1} \left(f\left( \frac{i}{N} \right) - f(1) \right) \left[ \int_0^t \int_{\mathbb{R}^+} \mathbb{1}_{\theta \leq r(i/N, X_N^N(s^-)) \, X_i^N(s^-)} \, \widetilde{\mathcal{N}}_i^r(ds, \, d\theta) \right. \\
        & \hspace{4 cm}+ \int_0^t \int_{\mathbb{R}^+} \mathbb{1}_{\theta \leq N \, m((i - 1)/N, X_N^N(s^-)) \, X_{i - 1}^N(s^-)} \, \widetilde{\mathcal{N}}_{i - 1}^m(ds, \, d\theta)\\
        & \hspace{4 cm} \left.- \int_0^t\int_{\mathbb{R}^+} \mathbb{1}_{\theta \leq N \, m(i/N, X_N^N(s^-)) \, X_i^N(s^-)} \, \widetilde{\mathcal{N}}_i^m(ds, \, d\theta) \right] \\
        =& \frac{1}{N} \sum_{i = 2}^{N - 1} \left(f\left( \frac{i}{N} \right) - f(1) \right) \int_0^t \int_{\mathbb{R}^+} \mathbb{1}_{\theta \leq r(i/N, X_N^N(s^-))\, X_i^N(s^-)} \, \widetilde{\mathcal{N}}_i^r(ds, \, d\theta) \\
        &+ \frac{1}{N} \left( f \left( \frac{2}{N} \right) - f(1) \right) \int_0^t \int_{\mathbb{R}^+} \mathbb{1}_{\theta \leq N \, m((1)/N, X_N^N(s^-)) \, X_{1}^N(s^-)} \, \widetilde{\mathcal{N}}_{1}^m(ds, \, d\theta) \\&+ \frac{1}{N^2} \sum_{i = 2}^{N - 1} \Delta_{1/N} f\left( \frac{i + 1}{N} \right) \int_0^t\int_{\mathbb{R}^+} \mathbb{1}_{\theta \leq N \, m(i/N, X_N^N(s^-)) \, X_i^N(s^-)} \, \widetilde{\mathcal{N}}_i^m(ds, \, d\theta)
    \end{align*}
    Given the independence of the compensated measures we obtain
    \begin{align*}
        \left< M^{N,f} \right>_t =& \frac{1}{N^2} \sum_{i = 2}^{N - 1} \left(f\left( \frac{i}{N} \right) - f(1) \right)^2 \int_0^t \int_{\mathbb{R}^+} \mathbb{1}_{\theta \leq r(i/N, X_N^N(s^-))\, X_i^N(s^-)} \, d\theta \, ds \\
        &+ \frac{1}{N^2} \left( f \left( \frac{2}{N} \right) - f(1) \right)^2 \int_0^t \int_{\mathbb{R}^+} \mathbb{1}_{\theta \leq N \, m((1)/N, X_N^N(s^-)) \, X_{1}^N(s^-)} \, d\theta \, ds \\&+ \frac{1}{N^4} \sum_{i = 2}^{N - 1} \left(\Delta_{1/N} f \left( \frac{i + 1}{N} \right) \right)^2 \int_0^t\int_{\mathbb{R}^+} \mathbb{1}_{\theta \leq N \, m(i/N, X_N^N(s^-)) \, X_i^N(s^-)} \, d\theta \, ds \\
        =& \frac{1}{N} \int_0^t \left< \mu_s^N, \left(f - f(1) \right)^2 r\left( \cdot, X_N^N(s) \right) \right> \, ds + \frac{1}{N^2} \int_0^t \left< \mu_s^N, \left(\Delta_{1/N} f \right)^2 \, m\left( \cdot, X_N^N(s) \right) \right> \, ds\\
        &+ \frac{1}{N} \left( f \left( \frac{2}{N} \right) - f(1) \right)^2 \int_0^t m\left(\frac{1}{N}, X_N^N(s) \right) \, X_{1}^N(s) \, ds.
    \end{align*}
    For $M^N$ we have that
    \begin{align*}
        M^N(t) =& \frac{1}{N} \int_0^t \int_{\mathbb{R}^+} \mathbb{1}_{\theta \leq N \, m((1)/N, X_N^N(s^-)) \, X_{1}^N(s^-)} \, \widetilde{\mathcal{N}}_{1}^m(ds, \, d\theta) \\
        &+ \frac{1}{N} \sum_{i = 2}^{N - 1} \int_0^t \int_{\mathbb{R}^+} \mathbb{1}_{\theta \leq r(i/N, X_N^N(s^-))} \, \widetilde{\mathcal{N}}_{i}^r(ds, \, d\theta) - \frac{1}{N} \int_0^t \int_{\mathbb{R}^+} \mathbb{1}_{\theta \leq N \, d \, X_{N}^N(s^-)} \, \widetilde{\mathcal{N}}_{N}^d(ds, \, d\theta),
    \end{align*}
    which implies, by similar arguments as before
    \begin{align*}
        \left< M^N \right>_t =& \frac{1}{N} \left[ \int_0^t m\left( \frac{1}{N}, X_N^N(s) \right) \, X_{1}^N(s) \, ds + \int_0^t \left< \mu^N_s, r(\cdot, X_N^N(s)) \right> \, ds + \int_0^t d \, X_{N}^N(s) \, ds \right].
    \end{align*}
    Additionally, we compute
    \begin{align*}
        \left< M_1^N, M^{N,f} \right>_t =& - \frac{1}{N^2} \left( f \left( \frac{2}{N} \right) - f(1) \right) \int_0^t \int_{\mathbb{R}^+} \mathbb{1}_{\theta \leq N \, m \left( \frac{1}{N}, \, X_N^N(s) \right) X_1^N(s)} \, d\theta \, ds \\
        =& - \frac{1}{N} \left( f \left( \frac{2}{N} \right) - f(1) \right) \int_0^t m \left( \frac{1}{N}, \, X_N^N(s) \right) X_1^N(s) \, ds.
    \end{align*}
    And we can find in the same way $\left< M^{N,f}, M^N \right>_t$ and $\left< M_1^N, M^N \right>_t$.
\end{proof}

We can estimate the first and second moments of the triplet $(X^N_1, \mu^N, X^N_n)$ uniformly on finite time intervals, which will be needed in the following proofs.

\begin{prop} \label{prop:moments}
    Under Assumption \ref{assp:rm}, for any $N > 0$  and $T > 0$ we have that 
    \begin{equation} \label{eq:EY}
       \mathbb{E} \left[ \sup_{t \leq T} \left\{X_1^N(t) + \left< \mu_t^N, 1 \right> + X_N^N(t)\right\} \right] \leq \mathbb{E} \left[ X_1^N(0) + \left< \mu_0^N, 1 \right> + X_N^N(0) \right] e^{\widehat{r}T},
    \end{equation} 
    and
    \begin{align} \label{eq:EY2}
        \mathbb{E} &\left[ \sup_{t \leq T} \left( X_1^N(t) + \left< \mu_t^N, 1 \right> + X_N^N(t) \right)^2 \right] \nonumber \\ &\leq \left[ \mathbb{E} \left[ \left( X_1^N(0) + \left< \mu_0^N, 1 \right> + X_N^N(0) \right)^2 \right] + \frac{T}{2N} \right] \exp \left( 2 \, \widehat{r} \left( 1 + \frac{1}{2N} \right) T \right).
    \end{align}
\end{prop}
\begin{proof}
    Let us prove (\ref{eq:EY2}). Similar computations can be done to prove (\ref{eq:EY}).
    We define the process $Y^N(t) = X_1^N(t) + \left< \mu_t^N, 1 \right> + X_N^N(t)$ and from \eqref{eq:stoch}, \eqref{eq:stoch2}, \eqref{eq:stoch3} and \eqref{empmes}, it follows that
    \begin{align*}
        Y^N(t) =& \, Y^N(0) + \frac{1}{N} \int_0^t \int_{\mathbb{R}^+} \mathbb{1}_{\theta \, \leq \, N \, r(1/N, X_N^N(s^-)) \, X_1^{N}(s^-)} \, \mathcal{N}_1^r(ds, \, d\theta) \\ 
        &+ \sum_{i = 2}^{N - 1} \frac{1}{N} \int_0^t \int_{\mathbb{R}^+} \mathbb{1}_{\theta \, \leq \, r(i/N, \, X_N^N(s^-)) \, X_i^{N} (s^-)} \, \mathcal{N}_i^r(ds, \, d\theta) \\
        &- \frac{1}{N} \int_0^t \int_{\mathbb{R}^+} \mathbb{1}_{\theta \, \leq \, d \, N \, X_N^{N}(s^-)} \, \mathcal{N}_N^d(ds, \, d\theta).
    \end{align*}
    For any $n \in \mathbb{N}^*$, we define the stopping time $T_n = \left\{ t > 0, \, Y^N(t) \geq n \right\}$ with respect to the filtration $\mathcal{F}_t$.
    Then, for $s \leq t$ we have
    \begin{align} \label{eq:Yp}
        &\left( Y^N(s \wedge T_n)  \right)^2 \nonumber\\
        & \qquad = \left( Y^N(0)  \right)^2 \nonumber\\
        & \qquad \, \,  \, \, \, + \int_0^{s \wedge T_n} \int_{\mathbb{R}^+} \left( \left( Y^N(s^-) + \frac{1}{N} \right)^2 - \left(Y^N(s^-) \right)^2 \right) \mathbb{1}_{\theta \, \leq \, N \, r\left(1/N, X_N^N(s^-) \right) \, X_1^{N}(s^-)} \,  \mathcal{N}_1^r(ds, \, d\theta) \nonumber \\ 
        & \qquad \, \,  \, \, \,  + \sum_{i = 2}^{N - 1} \int_0^{s \wedge T_n} \int_{\mathbb{R}^+} \left( \left( Y^N(s^-) + \frac{1}{N} \right)^2 - \left(Y^N(s^-) \right)^2 \right) \mathbb{1}_{\theta \, \leq \, r\left(i/N, X_N^N(s^-) \right) \, X_i^{N} (s^-)} \, \mathcal{N}_i^r(ds, \, d\theta) \\
        & \qquad \, \,  \, \, \,  + \int_0^{s \wedge T_n} \int_{\mathbb{R}^+} \left( \left( Y^N(s^-) - \frac{1}{N} \right)^2 - \left(Y^N(s^-) \right)^2 \right) \mathbb{1}_{\theta \, \leq \, N \, d \, X_N^{N}(s^-)} \, \mathcal{N}_N^d(ds, \, d\theta). \nonumber
    \end{align}
    By only considering the positive terms we obtain the upper bound
    \begin{align*}
        &\sup_{s \leq t} \left( Y^N(s \wedge T_n)  \right)^2 \\
        &\qquad \leq \left( Y^N(0)  \right)^2 \\
        & \qquad \, \, \, \, \, + \int_0^{t \wedge T_n} \int_{\mathbb{R}^+} \left( \left( Y^N(s^-) + \frac{1}{N} \right)^2 - \left(Y^N(s^-) \right)^2 \right) \mathbb{1}_{\theta \, \leq \, N \, r\left(1/N, X_N^N(s^-) \right) \, X_1^{N}(s^-)} \, \mathcal{N}_1^r(ds, \, d\theta) \\
        & \qquad \, \, \, \, \, + \sum_{i = 2}^{N - 1} \int_0^{t \wedge T_n} \int_{\mathbb{R}^+} \left( \left( Y^N(s^-) + \frac{1}{N} \right)^2 - \left(Y^N(s^-) \right)^2 \right) \mathbb{1}_{\theta \, \leq \, r\left(i/N, X_N^N(s^-) \right) \, X_i^{N} (s^-)} \, \mathcal{N}_i^r(ds, \, d\theta).
    \end{align*}
    Then, taking expectation and using Assumption \ref{assp:rm}, we obtain
    \begin{align*}
        \mathbb{E} \left[ \sup_{s \leq t} \left( Y^N(s \wedge T_n)  \right)^2 \right] \leq& \, \mathbb{E} \left[ \left( Y^N(0)  \right)^2 \right] +  \mathbb{E} \left[ \int_0^{t \wedge T_n} \left( 2 Y^N(s) + \frac{1}{N} \right) r\left(\frac{1}{N}, \, X_N^N(s) \right) X_1^{N}(s) \, ds \right] \\
        &+  \mathbb{E} \left[ \sum_{i = 2}^{N - 1} \int_0^{t \wedge T_n} \frac{1}{N} \left( 2 Y^N(s) + \frac{1}{N} \right) r\left(\frac{i}{N}, \, X_N^N(s) \right)  X_i^{N} (s) \, ds \right] \\
        \leq& \, \mathbb{E} \left[ \left( Y^N(0)  \right)^2 \right] +  2 \widehat{r} \, \mathbb{E} \left[ \int_0^{t \wedge T_n} \left( Y^N(s) + \frac{1}{2N} \right) X_1^{N}(s) \, ds \right] \\
        &+ 2 \widehat{r} \,  \mathbb{E} \left[ \int_0^{t \wedge T_n} \left( Y^N(s) + \frac{1}{2N} \right)  \left< \mu_s^N, 1 \right> \, ds \right] \\
        \leq& \, \mathbb{E} \left[ \left( Y^N(0)  \right)^2 \right] +  2 \widehat{r} \, \mathbb{E} \left[ \int_0^{t \wedge T_n} \left( Y^N(s) + \frac{1}{2N} \right) Y^{N}(s) \, ds \right]. \\
    \end{align*} 
    Using the fact that for any $x \in \mathbb{R}$ and $\varepsilon \geq 0$, $x^2 + \varepsilon x \leq (1 + \varepsilon) x^2 + \varepsilon$, we deduce that
    \begin{align*}
        &\mathbb{E} \left[ \sup_{s \leq t} \left( Y^N(s \wedge T_n)  \right)^2 \right] 
\\
       & \qquad  \, \mathbb{E} \left[ \left( Y^N(0)  \right)^2 \right] +  2 \widehat{r} \, \mathbb{E} \left[ \int_0^{t \wedge T_n} \left( \left( 1 + \frac{1}{2N} \right) \left( Y^N(s) \right)^2 + \frac{1}{2N} \right) \, ds \right] \\
        & \qquad 
        \leq \, \mathbb{E} \left[ \left( Y^N(0)  \right)^2 \right] + \frac{t}{2N} +  2 \widehat{r} \left( 1 + \frac{1}{2N} \right) \mathbb{E} \left[ \int_0^{t \wedge T_n} \left( Y^N(s) \right)^2 \, ds \right] \\
               & \qquad  \, \mathbb{E} \left[ \left( Y^N(0)  \right)^2 \right] + \frac{t}{2N} +  2 \widehat{r} \left( 1 + \frac{1}{2N} \right) \mathbb{E} \left[ \int_0^{t} \sup_{u \leq s} \left( Y^N(u\wedge T_n) \right)^2 \, ds \right].
    \end{align*}
    Since the process is bounded by $n + 1$ before $T_n$, Gronwall's lemma tells us that for any $T > 0$ 
    \begin{equation} \label{eq:supE}
         \mathbb{E} \left[ \sup_{s \leq T\wedge T_n} \left(Y^N(s) \right)^2\right] \leq \left( \mathbb{E} \left[ \left( Y^N(0)  \right)^2 \right] + \frac{T}{2N} \right) \exp \left( 2 \widehat{r} \left( 1 + \frac{1}{2N} \right) T \right).
    \end{equation}
    We have that $\left(T_n \right)_n$ goes to infinity almost surely, since otherwise there would exist $S > 0$ such that $\mathbb{P} \left( \sup_{n} T_n < S \right) >0$, which would imply $\mathbb{E} \left( \sup_{t \leq S \wedge T_n} (Y^N(t))^2 \right) \geq n^2 \mathbb{P}\left( \sup_{n} T_n < S \right)$ and this is a contradiction with (\ref{eq:supE}). Then, by Fatou's lemma
    \begin{equation*}
        \mathbb{E} \left[ \sup_{s \leq T} \left(Y^N(s)\right)^2 \right] < \left( \mathbb{E} \left[ \left( Y^N(0)  \right)^2 \right] + \frac{T}{2N} \right)  \exp \left( 2 \widehat{r} \left( 1 + \frac{1}{2N} \right) T \right).
    \end{equation*}
  This proves (\ref{eq:EY2}).
\end{proof} \ \\
Furthermore we can find a bound on the time average expected value of each immature compartment, which will be useful when studying the limit of System (\ref{eq:cont}). 
\begin{prop} \label{prop:Efin}
    For any $N > 0$, $T > 0$ and $2 \leq i \leq N - 1$, we have that
    \begin{align*}
        \int_0^T \mathbb{E} \left[ X_i^N(s) \right] ds \leq& \, \frac{1}{m_{min}} \left[ \mathbb{E} \left[ X_1^N(0) + \left< \mu_0^N, 1 \right> + X_1^N(0) \right] + \frac{\widehat{m}}{\widehat{r}} \, \mathbb{E} \left[ X_1^N(0) \right] \right] \exp \left( \widehat{r} T \right).
    \end{align*}
\end{prop}

\begin{proof}
    We have that
    \begin{align*}
        \frac{d}{dt}\mathbb{E} \left[ X_1^N(t) \right] =&  \mathbb{E} \left[\left( r\left(\frac{1}{N}, \, X_N^N(u)\right) - m\left(\frac{1}{N},  \, X_N^N(u)\right) \right) X_1^N(t) \right]  \, 
        \leq \, \widehat{r} \, \mathbb{E}\left[X_1^N(t) \right].
    \end{align*}
    Using Gronwall's Lemma, we obtain
    \begin{equation} \label{eq:X1}
        \mathbb{E} \left[ X_1^N(t) \right] \leq \mathbb{E} \left[ X_1^N(0) \right] \exp \left( \widehat{r} t \right).
    \end{equation}
    Considering any $2 \leq i \leq N - 1$, we can apply formula \eqref{eq:mXi} to all $2 \leq j \leq i$ and sum the corresponding equations. Then taking the expectations on both side, we infer that
    \begin{align*}
      &  \int_0^t \mathbb{E} \left[ m \left( \frac{i}{N}, \, X_N^N(s) \right) X_i^N(s) \right] \, ds \\ &\qquad =\frac{1}{N}\sum_{k = 2}^i \left[ \mathbb{E} \left[ X_k^N(0) \right] - \mathbb{E} \left[ X_k^N(t) \right] + \int_0^t \mathbb{E} \left[ r \left(\frac{k}{N}, \,  X_N^N(s) \right) X_k^N(s) \right] \, ds \right] \\
        &\qquad \qquad + \int_0^t \mathbb{E} \left[ m\left( \frac{1}{N}, \, X_N^N(s) \right) X_1^N(s) \right] \, ds.
    \end{align*}
    Since for any $1 \leq i \leq N$, the initial value $X_i^N(0)$ is non-negative, it follows that
    \begin{equation*}
        \frac{1}{N}\sum_{k = 2}^i \mathbb{E} \left[ X_k^N(0) \right] \leq Y^N(0) = X_1^N(0) + \left< \mu_0, 1 \right> + X_N^N(0).
    \end{equation*}
    This implies using (\ref{eq:EY}) and (\ref{eq:X1}) that
    \begin{align*}
      &  \int_0^t \mathbb{E} \left[ m \left( \frac{i}{N}, \, X_N^N(s) \right) X_i^N(s) \right] ds \\
        & \qquad  \leq \mathbb{E} \left[ Y^N(0) \right] +  \widehat{r} \, \mathbb{E}\left[ Y^N(0) \right] \int_{0}^{t} \exp(\widehat{r}s) \, ds + \widehat{m} \, \mathbb{E} \left[ X_1^N(0) \right] \int_0^t \exp \left( \widehat{r} s \right) \, ds \\
        &  \qquad \leq \left[ \mathbb{E} \left[ Y^N(0) \right]  + \frac{\widehat{m}}{\widehat{r}} \, \mathbb{E} \left[ X_1^N(0) \right] \right] \exp \left( \widehat{r} T \right).
    \end{align*}
    Since by Assumption \ref{assp:rm}, $m_{min} \leq m \left( i/N, X_N^N(s) \right)$, it yields that
    \begin{equation*}
        \int_0^t \mathbb{E} \left[ X_i^N(s) \right] ds \leq \frac{1}{m_{min}} \left[ \mathbb{E} \left[ Y_1^N(0) \right]  + \frac{\widehat{m}}{\widehat{r}} \,  \mathbb{E} \left[ X_1^N(0) \right] \right] \exp \left( \widehat{r} T \right).
    \end{equation*}  
    This ends the proof.
\end{proof}

\section{Approximation for a large number of compartments} \label{sec:cont}
\subsection{Tightness} \label{subsec:tight}

The goal now is to study the triplet $\left( X_1^{N}(t), \, \mu^N_t, \, X_N^{N}(t) \right)_{t\geq 0}$ as the number $N$ of compartments goes to infinity. Each empirical measure $\mu^N_t$ lives in the space $(M_F, w)$.
Some assumptions on the initial conditions are needed.
\begin{assp} \label{assp:lim}
    Assume that
    \begin{enumerate}
        \item The initial condition $\left( X^N_1(0), \, \mu_0^N, \, X^N_N(0)\right)$ is such that
        \begin{equation*}
            \sup_N \mathbb{E}\left[ \left( X_1^N(0) + \left< \mu_0^N, 1 \right> + X_N^N(0) \right)^2 \right] < \infty.
        \end{equation*} 
        \item The initial condition $\left( X^N_1(0), \, \mu_0^N, \, X^N_N(0)\right)$ converges in law (and in probability) in $\mathbb{R}^+ \times (M_F, w) \times \mathbb{R}^+$ to the deterministic triplet $(a_0, \, \mu_0, \, z_0)$, where $a_0 \in (0, \infty)$, $\mu_0 \in M_F$ and $z_0 \in [0, \infty)$.
    \end{enumerate}
\end{assp} \, \\
Notice that under this assumption, the estimates obtained in Propositions \ref{prop:moments} and \ref{prop:Efin} are now bounded uniformly in $N$. \\ \ \\
This allows us to now show the tightness of the sequence of laws of the triplet $\left( X_1^{N}, \, \mu^N, \, X_N^{N} \right)$ on any finite time interval. 
\begin{prop} \label{prop:tight}
    Under Assumptions \ref{assp:rm}-\ref{assp:lim}, for each $T > 0$ the sequence of laws of the triplet $(X_1^{N}(t), \, \mu^N_t, \, X_N^N(t), t\in [0,T])_{N}$ is tight in $\mathcal{P}( \mathbb{D}([0,T], \, \mathbb{R}^+) \times \mathbb{D}([0, T], \, (M_F, w)) \times \mathbb{D}([0,T], \, \mathbb{R}^+))$.
\end{prop}
\begin{proof}
    We want to show that
    \begin{itemize}
        \item the sequences of laws of $(X_1^{N}(t), \, t \in [0, T])_{N}$ and of $(X_N^{N}(t), \, t \in [0, T])_{N}$ are tight in $\mathcal{P} (\mathbb{D}([0, T]), \, \mathbb{R})$,
        \item the sequence of laws of the processes $\left( \mu^{N}_t, \, t \in [0, T] \right)_N$ is tight in $\mathcal{P} (\mathbb{D}([0, T]), \, (M_F, w))$,
    \end{itemize}
    which would prove the proposition. \\ \, \\
    
    We know that $X_1^{N} (t)$ is a semimartingale given by
    \begin{align*}
        X_1^{N}(t) =& \, X_1^{N}(0) + A^N_1(t) + M_1^{N}(t).
    \end{align*}
    As shown by Joffe and M\'etivier \cite{Tense}, in order to prove tightness of the sequence $\left( \mathcal{L} \left( X_1^{N} \right) \right)_{N}$ it is sufficient to show that the sequence $\left( \mathcal{L} \left( \sup_{t \leq T} \left| X_1^{N} (t) \right| \right) \right)_{N}$ is tight and that $A^N_1$ and $\left< M^N_1 \right>$ satisfy the Aldous condition found in \citet{Aldous}. \\ \, \\
    Since $X^N_1(t) \geq 0$, then, for $\eta > 0$, by Markov's inequality
    \begin{align*}
        \mathbb{P} \left( \sup_{t \leq T} \left| X_1^{N} (t) \right| \geq \eta \right) &\leq \frac{1}{\eta}\mathbb{E} \left[ \sup_{t \leq T} \left| X_1^{N} (t) \right| \right] = \frac{1}{\eta}\mathbb{E} \left[ \sup_{t \leq T} X_1^{N} (t) \right].
    \end{align*}
    Using Proposition \ref{prop:moments} and Assumption \ref{assp:lim}, this value goes uniformly to $0$ when $\eta \rightarrow \infty$. Hence, the sequence of laws of $\left( \sup_{t \leq T} |X_1^{N}| \right)_N$ is tight. \\ \, \\
    Let $\delta > 0$ and $\sigma$, $\tau$ be two stopping times such that $\sigma \leq \tau \leq (\sigma + \delta) \wedge T$. From (\ref{eq:A}a), by Markov's inequality and using Assumption \ref{assp:lim} we have
    \begin{align*}
        & \mathbb{P} \left( \left| A^N_1(\sigma) - A^N_1(\tau) \right| > \eta \right) \\
         & \qquad \quad  \leq \, \frac{1}{\eta} \, \mathbb{E} \left( \left| A^N_1(\sigma) - A^N_1(\tau) \right| \right) \\
        & \qquad  \quad  =  \, \frac{1}{\eta} \, \mathbb{E} \left[ \left| \int_{\sigma}^{\tau} \left( r\left(\frac{1}{N}, \, X_N^N(s) \right) - m\left(\frac{1}{N}, \, X_N^N(s) \right) \right)  X_1^{N}(s) \, ds \right| \right] \\
        & \qquad \qquad  \qquad \leq \, \frac{1}{\eta} \left( \widehat{r} + \widehat{m} \right) \, \mathbb{E} \left[ \int_{\sigma}^{\tau}  X_1^{N}(s) \, ds \right] \,
        \leq \, \frac{\delta}{\eta} \left( \widehat{r} + \widehat{m} \right) \sup_{N} \mathbb{E} \left[  \sup_{s \leq T}X_1^{N}(s) \right].
    \end{align*}
    By Proposition \ref{prop:moments} and Assumption \ref{assp:lim}, this value tends to 0 when $\delta \rightarrow 0$ uniformly in $N$. Hence, the process $A^N_1$ satisfy the Aldous condition. It's easy to prove in a similar way that this is also the case for $\left< M^N_1 \right>$. Then $\mathcal{L}(X_1^{N})$ is tight in $\mathcal{P} (\mathbb{D}([0, T], \mathbb{R}))$. \\ \, \\
    
    On the other hand, for the measure $\mu_t$ we have that for any $f \in \mathcal{C}([0, 1])$
    \begin{align*}
        \left< \mu_t^N, f \right> =& \left< \mu_0^N, f \right> + A^{N, f}(t) + M^{N,f}(t).
    \end{align*}
    In order to prove tightness of the sequence of laws of $\mu^N$ in $\mathcal{P}\left( \mathbb{D}([0, T]), \, (M_F, w) \right)$, as stated in \citet{Roelly}, it is sufficient to prove that for any $t \geq 0$ and each $f \in \mathcal{C}([0, 1])$, the sequence of laws of $\left< \mu^N, f \right>$ is tight in $\mathbb{D}([0, T], \, \mathbb{R})$, and that $A^{N, f}$ and $\left<M^{N,f}\right>$ satisfy the Aldous condition. We start by proving that for any continuous function $f$ on $[0, 1]$,
    \begin{equation*}
        \sup_N \mathbb{E} \left[ \sup_{t \leq T} \left| \left< \mu_t^N, f \right> \right| \right] < \infty.
    \end{equation*}
    This is an immediate consequence of Assumption \ref{assp:lim} and Proposition \ref{prop:moments} since $X_i^N(t)$ is non-negative for each $2 \leq i \leq N - 1$. It follows that
    \begin{equation*}
        \sup_N \mathbb{E} \left[ \sup_{t \leq T} \left| \left< \mu_t^N, f \right> \right| \right] \leq \, ||f||_\infty  \sup_N \mathbb{E} \left[  \sup_{t \leq T} \left< \mu_t^N, 1 \right> \right].
    \end{equation*}
    Let $\sigma$ and $\tau$ be two stopping times such that $\sigma \leq \tau \leq (\sigma + \delta) \wedge T$, then
    \begin{align} \label{eq:Mmu}
        \mathbb{E} \left[ \left| \left<  M^{N,f} \right>_{\sigma} - \left<  M^{N, f} \right>_{\tau} \right| \right] =& \, \mathbb{E} \left[ \left| \frac{1}{N} \left[ \int_\sigma^\tau \left< \mu_s^N, (f - f(1))^2 \, r\left(\cdot, \, X_N^N(s) \right) \right> \, ds \right. \right. \right. \nonumber\\
        &\ \, \, \, \, \, \, \, \, + \frac{1}{N} \int_\sigma^\tau \left< \mu_s^N, \left( \Delta_{1/N} f \right) \, m\left(\cdot, \, X_N^N(s) \right) \right> \, ds \nonumber \\
        & \, \, \, \, \, \, \, \, \left. \left. \left. + \left(  f \left(\frac{2}{N}\right) - f(1) \right)^2 \int_{\sigma}^{\tau} m\left(\frac{1}{N}, \, X_N^N(s) \right) X_1^N(s) \, ds \right] \right| \right] \nonumber\\
        \leq& \, C \, ||f||_\infty^2 \, \left[ \mathbb{E} \left[ \int_\sigma^\tau \left< \mu_s^N, 1 \right> \, ds \right] + \mathbb{E} \left[ \int_{\sigma}^{\tau} X_1^N(s) \, ds \right] \right] \nonumber\\
        \leq& \, \delta \, C ||f||_\infty^2\left[ \mathbb{E} \left[ \sup_{\sigma \leq s \leq \tau} \left< \mu_s^{N}, 1 \right> \right] + \mathbb{E} \left[ \sup_{\sigma \leq s \leq \tau} X_1^{N}(s) \right] \right].
    \end{align}
    Similarly, we find that
    \begin{align} \label{eq:Amu}
        \mathbb{E} \left[ \left|  A^{N, f}(\tau) -  A^{N, f}(\sigma) \right| \right] \leq& \, \delta \, C \, ||f||_\infty^2 \left[  \mathbb{E} \left[ \sup_{\sigma \leq s \leq \tau} \left< \mu_s^{N}, 1 \right> \right] + \mathbb{E} \left[ \sup_{\sigma \leq s \leq \tau} X_1^{N}(s) \right] \right].
    \end{align}
    By Proposition \ref{prop:moments}, and given that $f$ is bounded, $\left< M^{N, f} \right>$ and $A^{N, f}$ both satisfy the Aldous condition. Hence, all the conditions for the collection of laws $\left(\mu^{N}_t, \, t \in [0, T] \right)_N$ to be tight in $\mathcal{P} (\mathbb{D}([0, T]), \, (M_F, w))$ are satisfied. \\ \, \\
    
    Finally, we have that $X_N^{N} (t)$ is a semimartingale given by
    \begin{align*}
        X_N^{N}(t) =& \, X_N^{N}(0) + A^N_N(t) + M^{N}(t).
    \end{align*}
    As for $X_1^N$, to prove the tightness of the sequence $\left( \mathcal{L} \left( X_N^{N} \right) \right)_{N}$, we will prove that the sequence $\left( \mathcal{L} \left( \sup_{t \leq T} \left| X_N^{N} (t) \right| \right) \right)_{N}$ is tight and that $A^N_N$ and $\left< M^N \right>$ satisfy the Aldous condition. \\ \, \\
    The tightness of the sequence of laws of $(\sup_{t \leq T} \left| X_N^{N} (t) \right|)_N$ follows from Markov's inequality and previous estimates. \\ \, \\		
    Similar computations can be done as the ones for $X_1^N(t)$ to find that for $\delta > 0$ and $\sigma$, $\tau$ two stopping times such that $\sigma \leq \tau \leq (\sigma + \delta) \wedge T$, we have
    \begin{align*}
        \mathbb{P} \left( \left| A^N_N(\sigma) - A^N_N(\tau) \right| > \eta \right) \leq& \, \frac{1}{\eta} \, \left[ \mathbb{E} \left[ \left|\int_\sigma^\tau m\left( \frac{N - 1}{N}, \, X_N^N(s) \right) X_{N - 1}^N (u) \right| \right] + \delta \, d \, \mathbb{E} \left[ \sup_{\sigma \leq u \leq \tau} X_N^{N}(u) \right] \right],
    \end{align*}
    and
    \begin{align*}
        \mathbb{P} \left( \left| \left< M^N \right>_{\tau} - \left< M^N \right>_{\sigma} \right| > \eta \right) \leq& \, \frac{1}{N}\ \frac{1}{\eta} \, \left[  \mathbb{E} \left[ \int_\sigma^\tau m\left( \frac{N - 1}{N}, \, X_N^N(s) \right) X_{N - 1}^N (u) \right] + \delta \, d \, \mathbb{E} \left[ \sup_{\sigma \leq u \leq \tau} X_N^{N}(u) \right] \right].
     \end{align*}
    We use equation (\ref{eq:XN-1}) to deal with the term involving $X_{N - 1}^{N}$, by studying the expectation of the equation. We observe that
    \begin{align} \label{eq:EmXN}
       &  \mathbb{E} \left[ \int_\sigma^\tau m\left( \frac{N - 1}{N}, \, X_N^N(s) \right) X_{N -1}^N(s) \, ds \right] \\
        & \qquad \leq   \, \mathbb{E} \left[ \left|\left< \mu_{\sigma}^N, 1 \right> - \left< \mu_\tau^N, 1 \right> \right|\right] + \mathbb{E} \left[  \left| M^{N, 1}(\tau) - M^{N, 1}(\sigma) \right| \right] \nonumber \\
&        \qquad \qquad + \mathbb{E} \left[ \left| \int_\sigma^\tau \left< \mu_s^N, r \left( \cdot, \, X_N^N(s) \right) \right> \, ds + \int_\sigma^\tau m\left( \frac{1}{N}, \, X_N^N(s) \right) X_1^N(s) \, ds  \right| \right]. \nonumber
    \end{align}
    We study each term in the right hand side individually. By (\ref{eq:Mmu}) it yields that
    \begin{align*}
        \mathbb{E} \left[  \left| M^{N, 1}(\tau) - M^{N, 1}(\sigma) \right| \right] &\leq \left( \mathbb{E} \left[ \left< M^{N, 1} \right>_\tau - \left<M^{N, 1}\right>_\sigma \right] \right)^{1/2} \\&\leq \, \delta^{1/2} \, C \left[ \mathbb{E} \left[ \sup_{\sigma \leq s \leq \tau} \left< \mu_s^{N}, 1 \right> \right] + \mathbb{E} \left[ \sup_{\sigma \leq s \leq \tau} X_1^{N}(s) \right] \right]^{1/2}.
    \end{align*}
    Using this argument and (\ref{eq:Amu}) it follows that
    \begin{align*}
        \mathbb{E} \left[\left| \left< \mu_{\sigma}^N, 1 \right> - \left< \mu_\tau^N, 1 \right> \right|\right] \leq& \mathbb{E} \left[ \left|  A^{N, 1}(\tau) -  A^{N, 1}(\sigma) \right| \right] + \left(\mathbb{E} \left[ \left| \left<  M^{N,1} \right>_{\sigma} - \left<  M^{N, 1} \right>_{\tau} \right| \right] \right)^{1/2} \\ \leq& \, \delta \, C \left[  \mathbb{E} \left[ \sup_{\sigma \leq s \leq \tau} \left< \mu_s^{N}, 1 \right> \right] + \mathbb{E} \left[ \sup_{\sigma \leq s \leq \tau} X_1^{N}(s) \right] \right] \\&+ \left( \delta \, C \left[ \mathbb{E} \left[ \sup_{\sigma \leq s \leq \tau} \left< \mu_s^{N}, 1 \right> \right] + \mathbb{E} \left[ \sup_{\sigma \leq s \leq \tau} X_1^{N}(s) \right] \right] \right)^{1/2}.
    \end{align*}
    Finally, by Assumption \ref{assp:rm} we get that
    \begin{align*}
        &\mathbb E\left(\int_\sigma^\tau \left< \mu_s^N, r \left( \cdot, \, X_N^N(s) \right) \right> \, ds + \int_\sigma^\tau m\left( \frac{1}{N}, \, X_N^N(s) \right) X_1^N(s) \, ds \right) \\
        &\qquad \qquad \leq \delta \, C \left[ \mathbb{E} \left[ \sup_{\sigma \leq s \leq \tau} \left< \mu_s^N, 1 \right>  \right] + \mathbb{E} \left[ \sup_{\sigma \leq s \leq \tau} X_N^1(s)  \right] \right].
    \end{align*}
    Hence, by gathering these results and using Proposition \ref{prop:moments}, we obtain that (\ref{eq:EmXN}) tends to $0$ when $\delta \rightarrow 0$ uniformly in $N$, which implies the same for $\mathbb{P} \left( \left| A^N_N(\tau) - A^N_N(\sigma) \right| > \eta \right)$ and $\mathbb{P} \left( \left| \left< M^N\right>_{\tau} - \left< M^N \right>_{\sigma} \right| > \eta \right)$. Hence, the processes $A_N^N$ and $\left< M^N \right>$ satisfy the Aldous condition. Then, $\mathcal{L}(X_N^{N})$ is tight in $\mathcal{P} (\mathbb{D}([0, T], \, \mathbb{R}))$. This ends the proof.
\end{proof} \, \\

\subsection{Identification of the limit} \label{subsec:iden}
The tightness of the sequence of laws of $(X_1^{N}, \, \mu_t^N, \, X_N^N)_N$ guarantees the existence of converging subsequences. Let us now identify the system satisfied by these  limiting values. Recall that the initial condition $(a_0, \, \mu_0, \, z_0)$ defined in Assumption \ref{assp:lim} is deterministic. 
\begin{prop} \label{prop:iden}
    Under Assumptions \ref{assp:rm}-\ref{assp:lim}, let $T > 0$ and consider a subsequence $(X_1^{N_k}, \, \mu^{N_k}, \, X_{N_k}^{N_k})$ that converges in law in $\mathbb{D}([0, T], \, \mathbb{R}^+) \times \mathbb{D}([0, T], \, (M_f, w)) \times \mathbb{D}([0, T], \, \mathbb{R}^+)$ to a limit $(a, \, \mu, \, z)$ as $N_k \rightarrow \infty$. Then, this limit is deterministic, belongs to $\mathcal{C}([0, T], \mathbb{R}^+) \times \mathcal{C}([0, T], M_F) \times \mathcal{C}([0, T], \mathbb{R}^+)$, and for any function $f\in \mathcal{C}^1([0,1])$ and $t \in [0, T]$, it satisfies the system 
    \begin{align} \label{eq:lim}
    \begin{cases}
        &a(t) = \, a_0 + \int_0^t \left[r \left(0, z(s) \right) - m \left(0, z(s) \right) \right]  a(s) \, ds \, ; \\
        &\left< \mu_t, f \right> = \, f(0) \int_0^t m\left(0, z(s) \right) a(s) \, ds + \left< \mu_0, f \right> + \int_0^t \left< \mu_s, f \, r\left(\cdot, z(s) \right) \right> \, ds + \int_0^t \left< \mu_s, f' \, m\left(\cdot, z(s) \right) \right> \, ds\, ;  \\
        &\ \ \ \ \ \ \ \ \ \ \ \ - f(1) \left[ \left< \mu_0, 1 \right> - \left< \mu_t, 1 \right> + \int_0^t \left< \mu_s, r\left(\cdot, z(s) \right) \right> \, ds +  \int_0^t m\left(0, z(s) \right) a(s) \, ds  \right] \, ;\\
        &z(t) = \, z_0 +\left< \mu_0, 1 \right> - \left< \mu_t, 1 \right> + \int_0^t \left< \mu_s, r\left(\cdot, z(s) \right) \right> \, ds +  \int_0^t m\left(0, z(s) \right) a(s) \, ds - d \int_0^t z(s) \, ds \, ;\\
        &\mu_t \left( \left\{ 0 \right\} \right) = \mu_t \left( \left\{ 1 \right\} \right) = 0 \qquad \text{for almost every } t \in [0, T]\, .
        \end{cases}
    \end{align}
\end{prop} \, \\

\begin{proof}
    Let $\mathbb{D}^3 = \mathbb{D}([0, T], \, \mathbb{R}^+) \times \mathbb{D}([0, T], \, (M_f, w)) \times \mathbb{D}([0, T], \, \mathbb{R}^+)$. Let us consider a limiting value $(a, \, \mu, \, z)$ in $\mathbb{D}^3$, which is the limit in law of the subsequence $(X_1^{N_k}, \, \mu^{N_k}, \, X_{N_k}^{N_k})$ when $N_k \rightarrow \infty$. By simplicity, we will denote this subsequence by $(X_1^{N}, \, \mu^{N}, \, X_N^{N})$. Let us first show that $(a, \, \mu, \, z)$ has continuous paths almost surely. \\
    
    For $\nu \in \mathbb{D} ([0, T], \, (M_F, w))$, $t \geq 0$ and $x \in \mathbb{D} ([0, T], \, \mathbb{R}^+)$, we define  $| \Delta \nu_t| = || \nu_t  - \nu_{t^-} ||_{BL}$ and $\Delta |x(t)| = |x(t) - x(t^-)|$. Notice that the application
    \begin{align*}
        (x, \, \nu, \, z) \in \mathbb{D}^3 \mapsto \sup_{t \leq T} (| \Delta x(t) |, \, | \Delta \nu_t  |, \, | \Delta z(t)| ) \in (\mathbb{R}^+)^3
    \end{align*}
    is continuous. Moreover, we easily observe that
    \begin{equation*}
        |\Delta X_1^{N}(t) |\leq \frac{1}{N}, \quad |\Delta \mu_t^{N} |\leq \frac{1}{N}, \quad|\Delta X_{N}^{N}(t)| \leq \frac{1}{N},
    \end{equation*}
    and then $(a, \, \mu, \, z)$ has continuous paths almost surely. 

    To prove that the triplet $(a, \, \mu, \, z)$ is a solution of System (\ref{eq:lim}), let us consider $f \in \mathcal{C}^1([0, 1])$ and $T \geq 0$. We define for any $t \in [0, T]$, the functions $\Psi_t^i (x, \nu, y): \mathbb{D}^3 \rightarrow \mathbb{R}$, for $i = 1, 2, 3$ by
    \begin{align*}
        \Psi_t^1(x, \nu, y) =& \, x(t) - x_0 - \int_0^t \left(r(0, z(s)) - m(0, y) \right) x(s) \, ds\, ; \\
        \Psi_t^2(x, \nu, y) =& \, \left< \nu_t, f \right> - \left< \nu_0, f \right> - \int_0^t \left< \nu_s, f \, r(\cdot, y) \right> \, ds + \int_0^t \left< \nu_s, f' \, m(\cdot, y) \right> \, ds \\&
        + f(0) \int_0^t m(0, y) x(s) ds \\ 
        &+ f(1) \left[ \left< \nu_0, 1 \right> - \left< \nu_t, 1 \right> + \int_0^t \left< \nu_s, r(\cdot, y) \right> \, ds + \int_0^t m(0, y) x(s) \, ds \right] \, ; \\
        \Psi_t^3(x, \nu, y) =& \, y(t) - y_0 - \left< \nu_0, 1 \right> + \left< \nu_t, 1 \right> - \int_0^t \left< \nu_s, r(\cdot, y) \right> \, ds - \int_0^t m(0, y) x(s) \, ds  + d \int_0^t y(s) \, ds\, .
    \end{align*}
    We will prove that $\mathbb{E} \left[ \left| \Psi_t^i(a, \mu, z) \right| \right] = 0$ for $i \in \{ 1, 2, 3 \}$ and $t\in[0,T]$, which will give the desired result. \\ \, \\
    Since the sequence of triplets $( X_1^{N}, \mu^{N}, X_N^{N} )_N$ is $\mathcal{C}$-tight, and since the functions $\Psi^i$, $i=1, 2, 3$ are continuous on $\mathcal{C}([0, T], \, \mathbb{R}^+) \times \mathcal{C}([0, T], \, (M_F, w)) \times \mathcal{C}([0, T], \, \mathbb{R}^+)$, we can deduce that $\Psi_t \left( X_1^{N}, \mu^{N}, X_N^{N} \right)$ converges in law to $\Psi_t(a, \mu, z)$ as $N \rightarrow \infty$. Let us check that the sequences $(\Psi^i_t(X_1^{N}, \mu^{N}, X_N^{N}))_N$ are uniformly integrable to get convergence in $L^1$. \\ \, \\
    By Assumption \ref{assp:lim} and Proposition \ref{prop:tight}, and since $f \in \mathcal{C}^1([0, 1])$, it follows that for any $N$
    \begin{align*}
        \left| \Psi_t^1 \left( X_1^{N}, \mu^N, X_N^N \right) \right| \leq& \, X_1^{N} (t) + X_1^{N} (0) +  \left(\widehat{r} + \widehat{m}\right) \int_0^t X_1^{N} (s) \, ds \\
        \leq& \,  C \sup_{s \leq T} X_1^{N} (s)\, ; \\
        \left| \Psi_t^2(X_1^N, \mu^N, X_N^N) \right| \leq& \, C \, ||f||_\infty \left[ \left< \mu^N_t, 1 \right> + \left< \mu^N_0, 1 \right> + \int_0^t \left< \mu^N_s, 1 \right> \, ds + \int_0^t X_1^N(s) \, ds \right]\\ 
        \leq& \, C \, ||f||_\infty\sup_{s \leq T} \left( X_1^N(s) +  \left< \mu^N_s, 1 \right> \right) \, ;\\
        \left| \Psi_t^3 \left( X_1^N, \mu^N, X_N^{N} \right) \right| \leq& \, X_N^N(t) + X_N^N(0) +  \left< \mu_0, 1 \right> + \left< \mu_t, 1 \right> + \widehat{r} \int_0^t \left< \mu_s, 1 \right> \, ds + \widehat{m} \int_0^t X_1^N(s) \, ds \\
        &+ d \int_0^t X_N^N(s) \, ds \\
        \leq& C \sup_{s \leq T} \left(X_1^{N} (s) + \left< \mu^N_s, 1 \right> + X_N^{N} (s) \right).
    \end{align*}
    Then, 
    \begin{align*}
        \left| \Psi^1_t \left( X_1^{N}, \mu^N, X_N^N \right) \right| \mathbb{1}_{\left| \Psi^1_t \left( X_1^{N}, \mu^N, X_N^N \right) \right| > \eta} \leq& \, \frac{1}{\eta}\, \left| \Psi^1_t \left(  X_1^{N}, \mu^N, X_N^N \right) \right|^2 \leq \, \frac{C}{\eta}\, \sup_{s \leq T} \left( X_1^{N} (s) \right)^2 \, ; \\ 
        \left| \Psi^2_t \left( X_1^{N}, \mu^N, X_N^N \right) \right| \mathbb{1}_{\left| \Psi^2_t \left( X_1^{N}, \mu^N, X_N^N \right) \right| > \eta} \leq& \, \frac{1}{\eta}\,\left| \Psi^2_t \left(  X_1^{N}, \mu^N, X_N^N \right) \right|^2 \\
        \leq& \, \frac{C}{\eta}\, ||f||_\infty \sup_{s \leq T} \left( X_1^{N} (s) + \left< \mu^N_s, 1 \right> \right)^2 \, ; \\ 
        \left| \Psi^3_t \left( X_1^{N}, \mu^N, X_N^N \right) \right| \mathbb{1}_{\left| \Psi^3_t \left( X_1^{N}, \mu^N, X_N^N \right) \right| > \eta} \leq& \, \frac{1}{\eta}\,\left| \Psi^3_t \left(  X_1^{N}, \mu^N, X_N^N \right) \right|^2 \\
        \leq & \, \frac{C}{\eta}\, \sup_{s \leq T} \left( X_1^{N} (s) + \left< \mu^N_s, 1 \right> + X_N^N(s) \right)^2.
    \end{align*}
    From Proposition \ref{prop:moments} and Assumption \ref{assp:lim} and since $f$ is bounded, the sequences $(\Psi^1_t(X_1^{N}, \mu^N, X_N^N))_{N}$, $(\Psi^2_t(X_1^{N}, \mu^N, X_N^N))_{N}$ and $(\Psi^3_t(X_1^{N}, \mu^N, X_N^N))_{N}$ are uniformly integrable for any $t \in [0, T]$.\\  \\
    Combining convergence in law and uniform integrability of the subsequences, we obtain that
    \begin{align*}
        \mathbb{E} \left[ \left| \Psi^i_t(a, \mu, z) \right| \right] &= \lim_{N \rightarrow \infty} \mathbb{E} \left[ \left| \Psi^i_t(X_1^{N}, \mu^{N}, X_{N}^{N}) \right| \right] \leq 
        \lim_{N \rightarrow \infty} \left(\mathbb{E} \left[  \Psi^i_t(X_1^{N}, \mu^{N}, X_{N}^{N})^2\right] \right)^{1/2}.
    \end{align*}
    Recognizing the martingales 
    \begin{equation*}
        \Psi^1_t(X_1^{N}, \mu^{N}, X_N^{N}) = M_1^{N}(t), \quad  \Psi^2_t(X_1^{N}, \mu^{N}, X_N^{N}) = M^{N, f}(t), \quad  \Psi^3_t(X_1^{N}, \mu^{N}, X_N^{N}) = M^{N}(t)
    \end{equation*}
    from System \eqref{eq:cont}, Proposition \ref{prop:mart} shows that these martingales are of order $1/N$ and ensures the convergence of the expectations of each $\Psi^i$ to $0$ when $N \rightarrow \infty$. Then, we can finally conclude that any limiting value of the sequence $(X_1^{N}, \, \mu_t^N, \, X_N^N)_N$ is deterministic and is a solution of the system (\ref{eq:lim}). \\ \ \\  
    Recall that the subsequence $(X_1^{N} , \, \mu^{N}, \, X_N^{N})_{N}$ converges to a limit $(a, \, \mu, \, z)_t$. We will now proceed to prove that for any limiting value, we have almost surely, almost everywhere in $t \in [0,T]$, $\mu_t(\left\{ 0 \right\} ) = 0$ and $\mu_t(\left\{ 1 \right\} ) = 0$. \\ \, \\
    Let $\varepsilon, \, \delta\in (0, 1)$ and $f_\varepsilon, \, f_{\delta} \in \mathcal{C}([0, 1])$ be two continuous positive functions such that
    \begin{align*}
        f_{\varepsilon}(0) &= 1, \quad 	f_{\varepsilon}(x) = 0 \text{ for } x \geq \varepsilon, \quad \| f_\varepsilon \|_\infty \leq 1 \, ;\\
        f_{\delta}(1) &= 1, \quad 	f_{\delta}(x) = 0 \text{ for } x \leq 1 - \delta, \quad \| f_\delta \|_\infty \leq 1.
    \end{align*}
    We study $\left(\left< \mu_t, f_{\varepsilon} \right> \right)_{t \geq 0}$. Since by construction for any $2 \leq i \leq N - 1$ such that $\varepsilon \leq i/N$ we have $f_\varepsilon(i/N) = 0$, then
    \begin{align*}
        \int_0^T \mathbb{E} \left[ \left< \mu_s^{N}, f_{\varepsilon} \right> \right] \, ds = \int_0^T \mathbb{E} \left[ \frac{1}{N} \sum_{i = 2}^{\lceil N \varepsilon \rceil} X_i^{N}(s) \,  f_{\varepsilon} \left( \frac{i}{N} \right) \right] \, ds = \frac{1}{N} \sum_{i = 2}^{\lceil N \varepsilon \rceil} \int_0^T \mathbb{E} \left[ X_i^{N}(s) \right] \, ds.
    \end{align*}
    Using Proposition \ref{prop:Efin} and Assumption \ref{assp:lim} we obtain
    \begin{align} \label{major}
        \int_0^T \mathbb{E} \left[ \left< \mu_s^{N}, f_{\varepsilon} \right> \right] \, ds \leq \frac{C_T}{N} \left(\lceil N \, \varepsilon \rceil - 1 \right)
        \leq \varepsilon \, C_T. 
    \end{align}
    Since $\| f_\varepsilon \|_\infty \leq 1$, we have that
    \begin{equation*}
        0 \leq \left< \mu_t^{N}, f_\varepsilon \right> \leq \left< \mu_t^{N}, 1 \right>,
    \end{equation*}
    which implies that
    \begin{equation*}
        \int_0^t \left< \mu_s^{N}, f_\varepsilon \right> \, ds \, \mathbb{1}_{\int_0^t \left< \mu_s^{N}, f_\varepsilon \right> ds > \eta} \leq \frac{1}{\eta} \left( \int_0^t \left< \mu_s^{N}, f_\varepsilon \right> \, ds \right)^2 \leq \frac{1}{\eta} \left( \int_0^t \left< \mu_s^{N}, 1 \right> \, ds\right)^2.
    \end{equation*}
    As the right hand side has a bounded expectation with respect to $N$ for all $t \in [0, T]$ by Proposition \ref{prop:moments} and Assumption \ref{assp:lim}, we obtain the uniform integrability of the sequence $\left(\int_0^T \left< \mu_s^{N}, f_\varepsilon \right> ds\right)_{N}$. \\ \, \\
    Using that $\mu^{N}$ converges in law to $\mu$ in $\mathbb{D}([0, T], \, (M_F, w))$ as $N$ tends to infinity, then $\int_0^T \left< \mu_s^{N}, f_\varepsilon \right> ds$ converges in law to $\int_0^t\left< \mu_s, f_\varepsilon \right> ds$. By uniform integrability we get 
    \begin{equation*}
        \mathbb{E} \left[ \int_0^T \left< \mu_s, f_\varepsilon \right> \, ds \right] = \lim_{N \rightarrow \infty} \mathbb{E} \left[ \int_0^T \left< \mu_s^{N}, f_\varepsilon \right> \, ds \right].
    \end{equation*}
   By \eqref{major}, it follows that
    \begin{equation*}
        \mathbb{E} \left[ \int_0^T \left< \mu_s, f_\varepsilon \right> ds \right] \leq \varepsilon \, C_T.
    \end{equation*}
    Finally, when $\varepsilon \rightarrow 0$ we get by monotone limit
    \begin{equation*}
        \int_0^T \mu_s(\{0\}) \, ds = \lim_{\varepsilon \rightarrow 0} \int_0^T \left< \mu_s, f_{\varepsilon} \right> \, ds = 0 \quad a.s.
    \end{equation*} \, \\
    Hence,  almost everywhere in $t$, we have 
    \begin{equation*}
        \mu_t( \left\{ 0 \right\} ) = 0.
    \end{equation*}\ \\
    By applying similar computations to $\left< \mu_u^{N}, f_\delta \right>$, we find that
    \begin{equation*}
        \int_0^T \mathbb{E} \left[ \left< \mu_s^{N}, f_\delta \right> \right] \, ds \leq \delta \, C_T.
    \end{equation*}
    Therefore we get that almost everywhere in $t$,
    \begin{equation*}
        \mu_t( \left\{ 1 \right\} ) = 0.
    \end{equation*}
    This proves the proposition.
\end{proof}  \, \\

\subsection{Uniqueness} \label{subsec:uniq}
Having found the characterization of the limiting system \eqref{eq:lim}, we are interested in studying the uniqueness of its solution. We start by introducing a classical flow construction to adapt the characteristics approach which will allow us to prove uniqueness. \\

We consider $z\in \mathcal C^1(\mathbb R_+)$ and extend this function to $z\in \mathcal C^1(\mathbb R)$ in such a way that it remains bounded with bounded derivatives. Then 
for any $(t, x) \in \mathbb R^2$, we consider the solution $M^z(\cdot,t,x)$ of the (non-autonomous) Cauchy problem defined for $s\in \mathbb R$ by
\begin{equation}\label{eq:mat}
    \frac{\partial M^z}{\partial s} (s, t, x) = m(M^z(s, t, x), \, z(s)), \quad M^z(t, t, x) = x,
\end{equation}
where $m$ has been extended to a function of $\mathcal C^1(\mathbb R\times \mathbb R)$ in such a way that it remains bounded, with bounded derivative. The function $M^z(\cdot, t, x)$ represents the dynamics of the maturation of a cell whose level at time $t$ is given by $x$. 
The function $m$ is Lipschitz continuous in $x\in \mathbb R$ uniformly with respect to the other variable. Then we can use the classical Cauchy Lipschitz theory (see \citet{HirschSmale} Chapter 15) and obtain existence and uniqueness of the Cauchy problem $(\ref{eq:mat})$ on a maximal time interval. Adding that $m$ is bounded, the solution cannot explode in finite time. Thus, $M^z$ is well defined for any time and generates a flow in time and space on $\mathbb R^3$. This flow is also a diffeomorphism in space, which can be seen by reversing the flow in time. Observe that in our model, $m$ is bounded on $\mathbb R\times \mathbb R$ and the flow is also a diffeomorphism in time. More precisely, let us gather these useful and classical properties in the following statement.
\begin{prop} \label{prop:diff}
    \begin{enumerate}
  The following properties hold :
        \item[i)] For all $(t, x) \in \mathbb R^2$, System (\ref{eq:mat}) admits a unique solution $M^z(\cdot, t, x)$ on $\mathbb R$. It defines an application $ M^z: (s, t, x)\in  \mathbb{R} \times \mathbb{R} \times \mathbb R 
             \mapsto M^z(s, t, x) \in \mathbb{R}.$
        \item[ii)] For all  $t_1, t_2, t_3 \in \mathbb R$ and $x\in \mathbb R$,  
        \begin{equation}
        \label{eq:flow}
            M^z(t_1, t_3, x) = M^z (t_2, t_3, M^z(t_1, t_2, x)),
        \end{equation}
        \item[iii)] The function $M^z$ belongs to $\mathcal C^1(\mathbb R^3, \mathbb R)$.
        \item[iv)] For all $s,t \in \mathbb R$, the function $x \in \mathbb R \mapsto M^z(s, t, x) \in \mathbb R$ is a $\mathcal{C}^1$-diffeomorphism.
        \item[v)] For all $s \in \mathbb{R}$ and  $x \in \mathbb{R} $, the function $t \in \mathbb{R} \mapsto M^z(s, t, x) \in \mathbb{R}$ is a $\mathcal{C}^1$-diffeomorphism. The inverse function $(s, y, x) \in \mathbb{R}^3 \mapsto \kappa(s, y, x) \in \mathbb{R}$ satisfying  $M^z(s, \kappa(s, y, x) , x) = y$ on $\mathbb R^3$ belongs to $\mathcal{C}^1(\mathbb{R}^3, \mathbb{R})$.
        \item[vi)] For any $(s,t,x) \in \mathbb R^3$,
        \begin{align}
            \partial_s M^z(s,t,x) + m(x, z(s)) \, \partial_x M^z(s,t,x) = 0. 
        \end{align} 
    \end{enumerate}
\end{prop}
\begin{proof}
    The existence and uniqueness of point $i)$ for non-autonomous systems can be found in Chapter 15 of \cite{HirschSmale}. The composition property $ii)$ can be found in Chapter 8 of \cite{HirschSmale} for autonomous systems and extends directly to the non-autonomous case. For global existence and uniqueness and for points $iii)$ and $iv)$ in non-autonomous systems, we refer to Chapter 4, Theorem 3.3 of \citet{Mouhot} which covers our framework. We prove that the function $t \in \mathbb{R} \mapsto M^z(s, t, x)$ is a diffeomorphism (first part of $v)$) by differentiating \eqref{eq:flow} with respect to $t_2$ and taking $t_2 = t_3$:
    \begin{equation} \label{eq:flow2}
        \frac{\partial M^z}{\partial t} (s, t, x) = - \,m(M^z(s, t, x), \, z(s)),
    \end{equation} 
    and using the strict positivity of the function $m$. Continuity of $\kappa$ on $\mathbb R^3$ in point $v)$ follows from the implicit function theorem since $M^z(t, s, x) - y$ belongs to  $\mathcal{C}^1(\mathbb{R}^4, \mathbb{R})$ and $M^z(t, \kappa(s, y, x), x ) - y = 0$ and the differential of $M^z(t, s, x) - y$ with respect to $s$ is non zero due to the positivity of $m$. Finally, Theorem 3.6 in \cite{Mouhot} proves $vi)$.
\end{proof}

We complement the above properties on the flow $M$ by the following estimation, which will be useful for the proof of the uniqueness of the solution of \eqref{eq:lim}. 
\begin{prop} \label{prop:flow}
    For any $z, \hat{z} \in \mathcal{C}(\mathbb R_+)$, we have
    \begin{equation*} 
        |M^z(s, t, x) - M^{\widehat{z}}(s, t, x)| \leq \exp(L_m \, T) \int_s^t | z(u) - \widehat{z}(u)| \, du,
    \end{equation*}
    for all  $s, t \in [0, T]$ and $x \in \mathbb [0, 1]$.
    \end{prop} 
\begin{proof} 
    Using (\ref{eq:mat}) and the Lipschitz continuity of the function $m$, we have 
    \begin{align*}
        |M^z(s, t, x) - M^{\widehat{z}}(s, t, x)| =& \left| \int_s^t \left(m(M^z(u, t, x), \, z(u)) - m(M^{\widehat{z}}(u, t, x), \, \widehat{z}(u))\right) \, du \right| \\
        \leq& L_m \int_s^t \left[ |M^z(u, t, x) - M^{\widehat{z}}(u, t, x)| + | z(u) - \widehat{z}(u)|  \right] \, du.
    \end{align*} 
    Gronwall's Lemma allows to conclude.
\end{proof}
We can then prove the uniqueness of the solution of the system \eqref{eq:lim}. 
\begin{prop} \label{prop:uniquenessdiff}
    Under Assumption \ref{assp:rm}, for any $T > 0$ and any $(a_0, \, \mu_0, \, z_0) \in \mathbb{R}_+ \times M_F \times \mathbb{R}_+$, for any function $f \in \mathcal{C}^1([0, 1])$, the system \eqref{eq:lim} has a unique solution $(a, \mu, z) \in \mathcal{C}([0,T], \, \mathbb{R}) \times \mathcal{C}([0,T], \, (M_F, w)) \times \mathcal{C}([0,T], \, \mathbb{R})$.
\end{prop}

\begin{proof}
    Consider the measure that gathers the three elements of the triplet $(a, \mu, z)$ given for each $t \geq 0$ by
    \begin{equation} \label{eq:nu}
        \nu_t = a(t) \delta_0 + \mu_t + z(t) \delta_1,
    \end{equation}
    where $a(t)$, $\mu_t$, $z(t)$ are defined by System (\ref{eq:lim}). \\ \, \\
    Then, for any $f \in \mathcal{C}^1([0, 1])$ and any $t \geq 0$
    \begin{align*}
        \left< \nu_t, f \right> =& \, f(0) \, a(t) + \left< \mu_t, f \right> + f(1) \, z(t) \\
        =& \, \left< \nu_0, f \right> +  \int_0^t f(0) \, r(0, z(s)) \, a(s) \, ds + \int_0^t \left< \mu_s, f \, r \left( \cdot, z(s) \right) \right> \, ds \\
        &+ \int_0^t \left< \mu_s, f' \, m( \cdot, z(s)) \right> ds - d \int_0^t f(1) \, z(s)  \, ds.
    \end{align*}
    For every $t \in \mathbb{R}_+$, we can extend the last equation by classical arguments to a test function $\phi_t(s, x) \in \mathcal{C}^1([0, t] \times [0, 1], \mathbb{R})$ as follows: 
    \begin{align}
        \left< \nu_t, \phi_t(t, \cdot) \right> =& \left< \nu_0, \phi_t(0, \cdot) \right> + \int_0^t \phi_t(s, 0) \, r(0, z(s))\  a(s) \, ds + \int_0^t \left< \mu_s, \phi_t(s, \cdot) \, r \left( \cdot, z(s) \right) \right> \, ds \nonumber\\
        &+ \int_0^t \left< \mu_s, m(\cdot, z(s)) \partial_x \phi_t(s, \cdot) \, +\partial_s \phi_t(s, \cdot)  \right> \, ds
        - d \int_0^t \phi_t(s, 1) \, z(s) \, ds.\nonumber
    \end{align} \, \\
    
   Let $f \in \mathcal{C}^1([0,1])$. We extend $f$ to a function $\widetilde f \in \mathcal{C}^1_b(\mathbb{R})$. We define the test function $\phi$ for $t \geq 0$ and $x\in [0,1]$ by  $\phi_t(s, x) = \widetilde f(M^z(t, s, x))$. Using point $vi)$ of Proposition \ref{prop:diff}, we can observe that for any $s,t\in \mathbb R$ and $x\in [0,1]$,
    \begin{align*} 
        \phi_t(t, x) = f(x), \quad m(x, z(s)) \, \partial_x \phi_t(s, x)+\partial_s \phi_t (s, x) = 0. 
    \end{align*}
    Therefore, we obtain that $(\nu_t )_t$ is solution of the mild equation given by
    \begin{align} \label{eq:nuPsi}
        \left< \nu_t, f \right>  = \left< \nu_t, \widetilde f \right> =& \, \left< \nu_0, \widetilde f(M^z(t, 0, \cdot)) \right> + \int_0^t \widetilde f(M^z(t,s,0)) \, r(0, z(s)) \, a(s) \, ds \\
        &\qquad + \int_0^t \left< \mu_s, \widetilde f(M^z(t, s, \cdot)) \, r \left( \cdot, z(s) \right) \right> \, ds - d \int_0^t \widetilde f(M^z(t, s, 1)) \, z(s) \, ds.\nonumber
    \end{align}
    
    Let us now consider $g\in BL_1([0, 1])$. Then $g$ can be extended on $\mathbb{R}$ to a bounded Lipschitz continuous function $\widetilde g\in BL_1(\mathbb{R})$. Since any function in $BL_1(\mathbb{R})$ is the limit of a sequence of $\mathcal C^1_b(\mathbb{R})$-functions for the bounded pointwise convergence, (take for example convoluted functions of $f$ by smooth mollifiers), the mild equation above is also satisfied by any function $g\in BL_1$ and its bounded Lipschitz continuous extension $\widetilde g$. Let us prove that the mild equation \eqref{eq:nuPsi} has a unique solution $(a, \mu, z)$ in $\mathcal{C}([0,T], \, \mathbb{R}) \times \mathcal{C}([0,T], \, (M_F, w)) \times \mathcal{C}([0,T], \, \mathbb{R})$ which satisfies that  a.e. in $t$, $\mu_t(\{0\})= \mu_t(\{1\})=0$. \\
    Let $(a, \mu, z)$ and $(\widehat{a}, \widehat{\mu}, \widehat{z})$ be two solutions of System (\ref{eq:lim}) and $\nu_t$, $\widehat{\nu}_t$ as defined in (\ref{eq:nu}) with $\nu_0 = \widehat{\nu}_0$. Using (\ref{eq:nuPsi}), it follows that
    \begin{align*}
        | \left< \nu_t - \widehat{\nu}_t, \widetilde g \right>| \leq& \left| \left< \nu_0,\widetilde g(M^z(t, 0, \cdot)) - \widetilde g(M^{\widehat{z}}(t, 0, \cdot) \right> \right| \\
        & + \int_0^t \left| \widetilde g(M^z(t,s,0)) \, r(0, z(s)) \, a(s) - \, \widetilde g(M^{\widehat{z}}(t,s,0)) \, r(0, \widehat{z}(s)) \, \widehat{a}(s) \right| ds \\
        &+ \int_0^t \left| \left< \mu_s, \widetilde g(M^z(t, s, \cdot)) \, r \left( \cdot, z(s) \right) \right> - \left< \widehat{\mu}_s, \widetilde g(M^{\widehat{z}}(t, s, \cdot)) \, r \left( \cdot, \widehat{z}(s) \right) \right>  \right| ds\\
        &+d \int_0^t \left| \widetilde  g(M^{z}(t, s, 1)) \, z(s) - \widetilde  g(M^{\widehat{z}}(t, s, 1)) \, \widehat{z}(s) \right| ds.
    \end{align*}
    We study individually the different terms on the right hand side. For the first term, given Assumption \ref{assp:lim}, the Lipschitz continuity of $g$ and the estimation given in Proposition \ref{prop:flow} we derive that
    \begin{align*}
        \left| \left< \nu_0,\widetilde g(M^z(t, 0, \cdot)) - \widetilde g(M^{\widehat{z}}(t, 0, \cdot) \right> \right|  \leq& \left| \left< \nu_0, 1 \right> \right| \sup_{x \in [0, 1]} \left|\widetilde  g(M^z(t, 0, x)) - \widetilde g(M^{\widehat{z}}(t, 0, x) \right| \\
        \leq& C_T \int_0^t \left| z(s) - \widehat{z}(s) \right| ds.
    \end{align*}
    Using the Lipschitz continuity of $r$ and $m$ (Assumption \ref{assp:rm}), the Lipschitz continuity of $\widetilde g$, and Proposition \ref{prop:flow}, we prove that the second term is bounded as follows
    \begin{align*}
        \left| \int_0^t \right. \widetilde g(M^z&(t,s,0)) \left. \, r(0, z(s)) \, a(s) - \,\widetilde g(M^{\widehat{z}}(t,s,0)) \, r(0, \widehat{z}(s)) \, \widehat{a}(s) ds \right| \\
        =&  \int_0^t \bigg|\, \widetilde g(M^z(t,s,0)) \, r(0, z(s)) \, (a(s) -  \widehat{a}(s) ) + \widehat{a}(s) \, r(0, z(s)) \,  \big(\widetilde g(M^z(t,s,0)) - \widetilde g(M^{\widehat{z}}(t,s,0)) \big) \\
        & + \, \widehat{a}(s) \, \widetilde g(M^{\widehat{z}}(t,s,0)) \, \big( r(0, z(s)) - r(0, \widehat{z}(s)) \big)\,\bigg| \,ds  \\
        \leq& C_T \, ||\widetilde{g}||_\infty \left[  \int_0^t  \left|a(s) -  \widehat{a}(s) \right| ds + \int_0^t \left|  z(s) - \widehat{z}(s) \right| ds \right].
    \end{align*}
    By similar arguments, we obtain
    \begin{align*}
        \left| \int_0^t\Big( \left< \right. \right. \mu_s&, \left. \left. \widetilde g(M^z(t, s, \cdot)) \, r \left( \cdot, z(s) \right) \right> - \left< \widehat{\mu}_s, \widetilde g(M^{\widehat{z}}(t, s, \cdot)) \, r \left( \cdot, \widehat{z}(s) \right) \right>\Big) \, ds \right|\\
        \leq& \int_0^t \left| \,\Big(\left< \mu_s - \widehat{\mu}_s, \widetilde g(M^z(t, s, \cdot)) \, r \left( \cdot, z(s) \right) \right> + \left< \widehat{\mu}_s, (\widetilde g(M^z(t, s, \cdot)) - \widetilde g(M^{\widehat{z}}(t, s, \cdot))) \, r \left( \cdot, z(s) \right) \right> \right. \\
        &\ \ \ \ \ \ \ \ \ \left. + \left< \widehat{\mu}_s, \widetilde g(M^{\widehat{z}}(t, s, \cdot)) \, (r \left( \cdot, z(s) \right) - r \left( \cdot, \widehat{z}(s) \right) )\right>\Big) \,\right|\, ds  \\
        \leq& C_T \, ||\widetilde{g}||_\infty \int_0^t  \left(\left|\left< \mu_s - \widehat{\mu}_s, 1 \right> \right| + \left|  z(s) - \widehat{z}(s) \right| \right) ds .
    \end{align*}
    Finally, by Lipschitz continuity of $\widetilde g$ and using Proposition \ref{prop:flow}, the last term satisfies
    \begin{align*}
        \left|  \int_0^t \left( \widetilde g(M^{z}(t, s, 1)) \, z(s) - \widetilde g(M^{\widehat{z}}(t, s, 1)) \, \widehat{z}(s)\right) \, ds \right| \leq C_T \int_0^t |z(s) - \widehat{z}(s)| ds.
    \end{align*}
    Then, by gathering all these results, we conclude that
    \begin{equation*}
        ||\nu_t - \widehat{\nu}_t||_{BL} \leq C_T \, ||\widetilde{g}||_\infty \int_0^t ||\nu_s - \widehat{\nu}_s||_{BL} \, ds.
    \end{equation*}
    Hence, by boundedness of $\widetilde{g}$, Gronwall's Lemma ensures that $||\nu_t - \widehat{\nu}_t ||_{BL} = 0$. Since for almost all $t \in [0,T]$, the measures $\mu_t $ and $\widehat \mu_t$ don't charge the two singletons $\{0\}$ and $\{1\}$, we have $a(t)=\widehat a(t)$ and $z(t)= \widehat z(t)$ for all $t$ (invoking also the continuity of the functions) and thus $\mu_t = \widehat \mu_t$ for all $t \in [0,T]$. Uniqueness of the solution of \eqref{eq:lim} is proved.
\end{proof}

We can then conclude to the convergence of the system  $(X_1^{N}, \mu^N, X_N^{N})_N$ as stated in the following theorem.

\begin{thm} \label{thm:conv}
    Under Assumptions \ref{assp:rm}-\ref{assp:lim}, for each $T > 0$, the sequence of processes $\left(  X_1^{N}, \mu^N, X_N^{N} \right)_N$ converges in probability in $\mathbb{D}([0,T], \, \mathbb{R}_+) \times \mathbb{D}([0,T], \, (M_F, w)) \times \mathbb{D}([0,T], \, \mathbb{R}_+)$ to the unique deterministic triplet $\left(a, \, \mu, \, z \right) \in \mathcal{C}([0,T], \, \mathbb{R}) \times \mathcal{C}([0,T], \, (M_F, w)) \times \mathcal{C}([0,T], \, \mathbb{R})$ satisfying \eqref{eq:lim}.
\end{thm} 

\section{Existence of a density} \label{sec:den}
In this section, we assume that the initial condition $\mu_0$ admits a density with respect to the Lebesgue measure on $[0, 1]$. We will prove that this property propagates over time. We introduce the following assumption.\\

\begin{assp} \label{assp:den}
    Assume that the limiting initial measure $\mu_0$ is  absolutely continuous with respect to the Lebesgue measure on $[0, 1]$.
\end{assp}

Then, the following result holds.

\begin{prop} \label{prop:den}
Under Assumptions \ref{assp:rm}-\ref{assp:lim}-\ref{assp:den}, for any $t \in [0, T]$, the measure $\mu_t$ defined in System \eqref{eq:lim} admits a density $u(t,\cdot)$ with respect to the Lebesgue measure.
\end{prop}

\begin{proof}
Consider the measure that gathers the immature and mature cells given by
\begin{equation*}
    \omega_t = \mu_t + z(t) \delta_1.
\end{equation*}
Then, using \eqref{eq:lim} we have that for any function $f \in \mathcal{C}^1([0,1])$
\begin{align*}
    \left< \omega_t, f \right> =& \, \left< \omega_0, f \right> +  \int_0^t f(0) \, m(0, z(s)) \, a(s) \, ds + \int_0^t \left< \mu_s, f \, r \left( \cdot, z(s) \right) \right> \, ds \\
    &+ \int_0^t \left< \mu_s, f' \, m( \cdot, z(s)) \right> ds - d \int_0^t f(1) \, z(s)  \, ds.
\end{align*}
Similarly as for \eqref{eq:nuPsi}, we consider the flow \eqref{eq:mat} and study the measure $\omega_t$ with the test function $\phi(s, x) = \widetilde{f} (M^z(s, t, x))$ where $\widetilde{f}$ is an extension of $f$ in $\mathcal{C}_b^1(\mathbb{R})$. Then, we obtain that $\omega_t$ is a solution of
\begin{align} \label{eq:omega}
    \left< \omega_t, f \right>  = \left< \omega_t, \widetilde f \right> =& \, \left< \omega_0, \widetilde f(M^z(t, 0, \cdot)) \right> + \int_0^t \widetilde f(M^z(t,s,0)) \, m(0, z(s)) \, a(s) \, ds \\
    &\qquad + \int_0^t \left< \mu_s, \widetilde f(M^z(t, s, \cdot)) \, r \left( \cdot, z(s) \right) \right> \, ds - d \int_0^t \widetilde f(M^z(t, s, 1)) \, z(s) \, ds.\nonumber
\end{align}
We study individually the terms in the right hand side. 

Assumption \ref{assp:den} implies the existence of a measurable function $u_0$ such that $\mu_0(dx) = u_0(x) dx$. We use now that $x \in (0,1)\rightarrow M^z(t, 0, x) \in (M^z(t, 0, 0), M^z(t, 0, 1))$ is a $C^1$ diffeomorphism (for any fixed $t$) with inverse function denoted by $h(t, \cdot)$, see Proposition \ref{prop:diff}.  We deduce that for each $t\in[0,T]$,
\begin{align*}
    \left< \mu_0, \widetilde f(M^z(t, 0, \cdot)) \right> &= \int_{(0, 1)} u_0(x) \, \widetilde{f}(M^z(t, 0, x)) \, dx \\
    &= \int_{(M^z(t, 0, 0), M^z(t, 0, 1))} u_0(h(t,y)) \, \widetilde{f}(y) \, \frac{\partial}{\partial y}  h(t, y) \, dy.
\end{align*}

Moreover
\begin{align*}
    \int_0^t \widetilde f(M^z(t, s, 0)) \, m(0, z(s)) \, a(s) \, ds  = -\int_0^{M^z(t, 0, 0)} \widetilde f(y) \, m(0, z(\tau(t,y))) \, a(\tau(t,y)) \, \frac{\partial}{\partial y}  \tau (t,y) dy
\end{align*}
where 
$s\rightarrow M^z(t,s,0)
$ is a $C^1$-diffeomorphism from $ (0,t)$ into its image by Proposition \ref{prop:diff} and $\tau(t,.)$ denotes its inverse function. Notice that the bounds of the integration interval have been exchanged since for any $t, x \in \mathbb{R}$, the application $s \mapsto M^z(t, s, x)$ is decreasing (by \eqref{eq:flow2} since $m$ is positive). \\

We now justify a change of variables for the term $ \int_0^t \left< \mu_s, \widetilde f(M^z(t, s, \cdot)) \, r \left( \cdot, z(s) \right) \right> \, ds$. 
We have defined in Proposition \ref{prop:diff}  the inverse function $y \rightarrow \kappa(t,x,y)$ of  $s \rightarrow M^z(t,s,x)$, and $\kappa$ belongs to $\mathcal{C}^1(\mathbb{R}^3,\mathbb{R})$. Note also  that for any $(t, x) \in [0, T] \times [0, 1]$ and any $y \in [x, M^z(t,s,x)]$, we have $\kappa(t,y,x) \in [0, t]$.\\
  
We will use indifferently the notations $\mu_s(dx)$ or $\mu(s,dx)$. For a fixed $t \in [0, T]$, we define  $\widetilde{\mu}$ by
\begin{equation*}
    \widetilde{\mu}(y, dx) = \mu(\kappa(t,y,x), dx).
\end{equation*}
We will  now  show that
\begin{align}
\label{eq:changemeasure}
    \int_0^t \int_0^1 & \widetilde{f}(M^z(t, s, x)) \, r \left( x, z(s) \right) \, \mu(s, dx) \, ds \nonumber \\
    &= -\int_\mathbb{R} \int_0^1 \mathbb{1}_{x \leq y \leq M^z(t, 0, x)} \widetilde f(y) \, r \left( x, z(\kappa(t, y, x)) \right) \, \frac{\partial}{\partial y} \kappa(t, y, x) \, \widetilde{\mu}(y, dx) \, dy.
\end{align}
In that aim, we will use approximating discrete measures. By Lemma \ref{lem:muk} in Appendix, given that the functions $f$, $M^z$, $r$, $z$ are all continuous and the composition of continuous functions remains continuous, it follows that
\begin{align*}
    \int_0^t \int_0^1 \widetilde f(M^z(t, s, x)) \, r \left( x, z(s) \right) \, \mu(s, dx) \, ds = \lim_{k \rightarrow \infty} \int_0^t \int_0^1 \widetilde f(M^z(t, s, x)) \, r \left( x, z(s) \right) \, \mu^k(s, dx) \, ds
\end{align*}
where $\mu^k$ is defined as in Appendix \eqref{def:muk} by $$
        \mu^k(s, dx) = \sum_{i \in \mathbb{Z}} \mu\left( s, \left( \frac{i - 1}{k}, \frac{i}{k} \right] \right) \delta_{\frac{i}{k}}.
   $$
 For a fixed $k \in \mathbb{N}$, for each $0 \leq i \leq k$ and $x = i/k$, we consider the change of variable $y = M^z(t,s,x)$ $(s=\kappa(t,x,y))$ to obtain
\begin{align*}
    \int_0^t &\widetilde f(M^z(t, s, x)) \, r \left( x, z(s) \right) \, \mu\left(s, \left( x - \frac{1}{k}, x\right] \right) \, ds \\
    &= - \int_\mathbb{R}  \mathbb{1}_{x \leq y \leq M^z(t, 0, x)} \widetilde f(y) \, r \left( x, z(\kappa(t, y, x)) \right) \, \frac{\partial}{\partial y} \kappa(t, y, x) \, \mu\left(\kappa(t, y, x), \left( x - \frac{1}{k}, x\right] \right) \, dy.
\end{align*}
Then, if we sum over all $0 \leq i \leq k$, it follows that
\begin{align*}
    \int_0^t \int_0^1 &\widetilde f(M^z(t, s, x)) \, r \left( x, z(s) \right) \, \mu^k(s, dx) \, ds \\
    &= - \int_\mathbb{R} \int_0^1 \mathbb{1}_{x \leq y \leq M^z(t, 0, x)} \widetilde f(y) \, r \left( x, z(\kappa(t, y, x)) \right) \, \frac{\partial}{\partial y} \kappa(t, y, x) \, \widetilde{\mu}^k(y, dx) \, dy.
\end{align*}
Then, by continuity of $\kappa(t, y, x)$ with respect to $(y, x)$  (Proposition \ref{prop:diff} $v)$), Lemma \ref{lem:muk} in Appendix yields that
\begin{align*} 
    \int_0^t \int_0^1 & \widetilde{f}(M^z(t, s, x)) \, r \left( x, z(s) \right) \, \mu(s, dx) \, ds \\
    =& \lim_{k \rightarrow \infty} \int_0^t \int_0^1 \widetilde f(M^z(t, s, x)) \, r \left( x, z(s) \right) \, \mu^k(s, dx) \, ds \\
    &= -\lim_{k \rightarrow \infty} \int_\mathbb{R} \int_0^1 \mathbb{1}_{x \leq y \leq M^z(t, 0, x)} \widetilde f(y) \, r \left( x, z(\kappa(t, y, x)) \right) \, \frac{\partial}{\partial y} \kappa(t, y, x) \, \widetilde{\mu}^k(y, dx) \, dy \\
    &= -\int_\mathbb{R} \int_0^1 \mathbb{1}_{x \leq y \leq M^z(t, 0, x)} \widetilde f(y) \, r \left( x, z(\kappa(t, y, x)) \right) \, \frac{\partial}{\partial y} \kappa(t, y, x) \, \widetilde{\mu}(y, dx) \, dy.
\end{align*}

We gather the results of each term of \eqref{eq:omega} to conclude that
\begin{align*}
     \int_0^1 \widetilde{f}(x)\, \omega_t(dx) = f(M^z(t, 0, 1)) z_0 + \int_\mathbb{R} \widetilde{f}(y)\, G(y, t)\, dy,
\end{align*}
where
\begin{align*}
    G(y, t) =& \mathbb{1}_{M^z(t, 0, 0) < y < M^z(t, 0, 1))} \, u_0(h(t,y)) \, \frac{\partial}{\partial y}  h(t, y) - \mathbb{1}_{0 \leq y \leq M^z(t, 0, 0)} \, m(0, z(\tau(t,y)) \, a(\tau(t,y)) \, \frac{\partial}{\partial y}  \tau (t,y) \\&- \int_0^1 \mathbb{1}_{x \leq y \leq M^z(t, 0, x)}  r \left( x, z(\kappa(t, y, x)) \right) \, \frac{\partial}{\partial y} \kappa(t, y, x) \, \widetilde{\mu}(y, dx).
\end{align*}

Let $f \in \mathcal{C}([0, 1])$ be a function such that $f(1) = 0$. We extend it to a function $\widetilde{f} \in \mathcal{C}_b^1(\mathbb{R})$ such that $\widetilde{f}(x) = 0$ for any $x \in [1, \infty)$. Notice that $M^z(0, t, 1) \geq 1$ since $M^z(0, 0, 1) = 1$ and by definition in \eqref{eq:mat}, it is strictly increasing with respect to $t$ since $m$ is positive by Assumption \ref{assp:rm}. Then, we obtain
\begin{equation*}
    \int_0^1 \widetilde{f}(x)\, \nu_t(dx) = \int_0^1 \widetilde{f}(x)\, \mu_t(dx) = \int_\mathbb{R} \widetilde{f}(y)\, G(y, t)\, dy.
\end{equation*}
Thus, we conclude that for every $t \in [0, T]$, $\mu_t(dx)$ admits  on $[0,1]$ the density $u(t,x)=G(x,t)$.
\end{proof}

Let us now prove that the triplet $(a, u, z)$  is  weak solution of a partial differential equations system (PDES).

\begin{prop}
    Under Assumptions \ref{assp:rm}-\ref{assp:lim}-\ref{assp:den}, the function $(a, u, z)$ is weak solution of the PDES given by
    \begin{subequations} \label{eq:pde}    \begin{alignat}{3} 
            \frac{d}{dt}a(t) &= \, \left[r \left(0, z(t) \right) - m \left(0, z(t) \right) \right]  a(t) \\
            \partial_t u(x,t) + \partial_x \left[ m(x, z(t)) \, u(x,t) \right] &= \,r(x, z(t)) \, u(x, t) \\
            \frac{d}{dt}z(t) &= \, m(1, z(t)) \, u(1, t) - d z(t), 
        \end{alignat}
    \end{subequations}
 for   $t\geq 0$ and $x\in [0,1]$,   with boundary conditions
    \begin{align*}
        &u(0, t) =  \, a(t) \\
        &u(1, t) = \, \frac{1}{m(1,z(t))} \left[ \int_0^1 \frac{d}{dt}u(x, t) \, dx - \int_0^1 r\left(x, z(t) \right) \, u(x, t) \, dx - m\left(0, z(t) \right) a(t) \right],
    \end{align*}
    and initial conditions $(a_0, \, z_0) \in (\mathbb{R}_+)^2$.
\end{prop}
\begin{proof}
By writing  $\mu_t(dx) = u(x, t) \, dx$ in System  (\ref{eq:lim}) and applying integration by parts to the term involving $f'$, it follows that
\begin{align*} 
   & \int_0^1 f(x) \, u(x, t) \, dx \\
    &  = \, f(0) \int_0^t m\left(0, z(s) \right) a(s) \, ds + \int_0^1 f(x) \, u(x, 0) \, dx + \int_0^t \int_0^1 f(x) \, r\left(x, z(s) \right) \, u(x, s) \, ds \\
    &+ \int_0^t \int_0^1 f'(x) \, m\left(x, z(s) \right) u(x, s) \, ds \\
    &- f(1) \left[ \int_0^1 u(x, 0) \, dx - \int_0^1 u(x, t) \, dx + \int_0^t \int_0^1 r\left(x, z(s) \right) \, u(x, s) \, ds +  \int_0^t m\left(0, z(s) \right) a(s) \, ds  \right] \\
    & = \, \int_0^1 f(x) \, u(x, 0) \, dx + \int_0^t \int_0^1 f(x) \, r\left(x, z(s) \right) \, u(x, s) \, ds - \int_0^t \int_0^1 f(x) \, \partial_x \left( m\left(x, z(s) \right) u(x, s) \right) \, ds \\
    &+ f(0) \left[ \int_0^t m\left(0, z(s) \right) a(s) \, ds -\int_0^t  m\left(0, z(s) \right) u(0, s) \, ds \right]\\
    &- f(1) \left[ \int_0^1 u(x, 0) \, dx - \int_0^1 u(x, t) \, dx + \int_0^t \int_0^1 r\left(x, z(s) \right) \, u(x, s) \, ds +  \int_0^t m\left(0, z(s) \right) a(s) \, ds \right. \\
    & \ \ \ \ \ \ \ \ \ \ \ \left. + \int_0^t m\left(1, z(s) \right) u(1, s)  \, ds \right]. \\
\end{align*}
In order to obtain a PDE we need to set the following boundary conditions for any $t \in [0, T]$
\begin{align*}
    \int_0^t m\left(0, z(s) \right) a(s) \, ds &= \int_0^t  m\left(0, z(s) \right) u(0, s) \, ds \\
    \int_0^t m\left(1, z(s) \right) u(1, s)  \, ds &= -\int_0^1 u(x, 0) \, dx + \int_0^1 u(x, t) \, dx\\
    &\qquad - \int_0^t \int_0^1 r\left(x, z(s) \right) \, u(x, s) \, ds - \int_0^t m\left(0, z(s) \right) a(s) \, ds.
\end{align*}
Then, we obtain the equation for $u(x, t)$ given by
\begin{equation*}
    \partial_t u(x,t) + \partial_x \left[ m(x, z(t)) \, u(x,t) \right] = r(x, z(t)) \, u(x, t).
\end{equation*}\\
Finally, using the fact that by assumption $0 < m_{min} \leq m(x, s)$ for any $(x, s) \in [0, 1] \times \mathbb{R}^+$,  we have that
\begin{align*}
    &\int_0^t m\left(0, z(s) \right) a(s) \, ds = \int_0^t  m\left(0, z(s) \right) u(0, s) \, ds \\
    \implies& \int_0^t m\left(0, z(s) \right) (a(s) - u(0, s)) \, ds = 0 \\
    \implies& m\left(0, z(s) \right) (a(s) - u(0, s)) = 0, \quad \forall s \leq t, \\
    \implies&  a(s) - u(0, s) = 0, \quad \forall s \leq t.
\end{align*}
This ends the proof. 
\end{proof}
\ \\ 
\textbf{Remark}: This system is equivalent to the one found by \cite{Doumic}, and our result proves as corollary the uniqueness of the system obtained in \cite{Doumic}, and subsequently to conclude a convergence result. Indeed, their limiting system writes
\begin{align} \label{eq:marie}
    \frac{d}{dt}a(t) &= [2 \alpha(z(t)) - 1] \, r(0, z(t)) \, a(t) \nonumber \\
    \partial_t u(x, t) + \partial_x \left[ m(x, z(t)) \, u(x, t) \right] &= r(x, z(t)) \, u(x, t) \\
    m(0, z(t)) \, u(0, t) &= 2[1 - \alpha(z(t))] \, r(0, z(t)) \, a(t) \nonumber \\
    \frac{d}{dt} z(t) &= m(1, z(t)) \, u(1, t) - d \, z(t), \nonumber
\end{align}
where $u(x, t)$ corresponds to the density of cells of maturation stage $x$ at time $t$ and $\alpha(z(t))$ is the fraction of progeny that stays at stem cell level. The term $[2 \alpha(z(t)) - 1] \, r(0, z(t))$ gathers all the changes in size in the stem cell population, which is equivalent in our case to $r(0, z(t)) - m(0, z(t))$. The term $2[1 - \alpha(z(t))] \, r(0, z(t))$ corresponds to the differentiation rate of the stem cell population, which we consider as $m(0, z(t))$. Finally, the weak formulation of the PDE \eqref{eq:marie} in the system allows us to find the boundary condition for $m(1, z(t)) \, u(1, t)$. For a function $f \in \mathcal{C}^1([0, 1])$ we have that
\begin{align*}
    \int_0^1 u(x, t) \, f(x) \, dx =&  \int_0^1 u(x, 0) \, dx - \int_0^t \int_0^1 \partial_x \left[ m(x, z(s)) \, u(x, s) \right] \, f(x) \, dx \, ds \\
    &+ \int_0^t \int_0^1 r(x, z(s)) \, u(x, s) \, f(x) \, dx \, ds \\
    =& \int_0^t f(0) \, m(0, z(s)) \, u(0, s) - f(1) \, m(1, z(s)) \, u(1, s) ds \\
    &+ \int_0^t \int_0^1  m(x, z(s)) \, u(x, s) \, f'(x) \, dx \, ds + \int_0^t \int_0^1 r(x, z(s)) \, u(x, s) \, f(x) \, dx \, ds.
\end{align*}
Then,
\begin{align*}
    \int_0^1 f(1) \, m(1, z(s)) \, u(1, s) ds =& \int_0^1 u(x, 0) \, dx - \int_0^1 u(x, t) \, f(x) \, dx + \int_0^t f(0) \, m(0, z(s)) \, u(0, s) ds \\
    &+ \int_0^t \int_0^1  m(x, z(s)) \, u(x, s) \, f'(x) \, dx \, ds + \int_0^t \int_0^1 r(x, z(s)) \, u(x, s) \, f(x) \, dx \, ds.
\end{align*}
For $f\equiv  1$, we obtain
\begin{align*}
    \int_0^1 m(1, z(s)) \, u(1, s) ds =& \int_0^1 u(x, 0) \, dx - \int_0^1 u(x, t) \, dx + \int_0^t m(0, z(s)) \, u(0, s) ds \\
    & + \int_0^t \int_0^1 r(x, z(s)) \, u(x, s) \, dx \, ds,
\end{align*}

\section{Numerical simulation} \label{sec:sim}
We performed numerical simulations for both the initial System \eqref{eq:stoch} with different numbers of compartments ($N = 50, 200$) and the limiting System \eqref{eq:pde} under the existence of a density. The stochastic system is simulated by classical acceptance-rejection procedure following the stochastic differential system \eqref{eq:stoch} where time has been discretized.
For the PDE part of the deterministic system, the numerical scheme is based on the following rules: the derivatives are approximated using finite backward differences, the PDE is approximated using an upwind method as seen in page 73 of \cite{Leveque} and the integrals are approximated using Riemann sums.

Our aim in these simulations is to illustrate the dynamics of the amplification through time as the maturity level increases. To do so, we started with an initial condition with only very few HSCs. Given the lack of available data on the differentiation and division rates, we considered arbitrary constant differentiation and proliferation functions $m$ and $r$. We then set $m = 0.02$, $r = 0.015$ and $d = 0.005$ for the simulations and consider an initial condition of $50$ stem cells.

In Figure \ref{fig:disc} and Figure \ref{fig:pde} we observe the distribution through time of immature cells depending on the maturation level for both systems. Figure \ref{fig:pde} also shows the dynamics of the mature cells population. Figure \ref{fig:disc} illustrates the average of $50$ simulations of the discrete system with $N = 50$ and $N = 200$ compartments while Figure \ref{fig:pde} illustrates the deterministic system. In both figures, we observe a propagation front that progressively fills all maturation stages and we observe an amplification in the number of cells as the maturation level grows, which is coherent with biological observations. The dynamics of the mature cells population in Figure \ref{fig:pde} further exemplifies the amplification as the number of mature cells is much higher than the one of stem cells.

\begin{figure}[h]
    \centering
    \includegraphics[width = \textwidth]{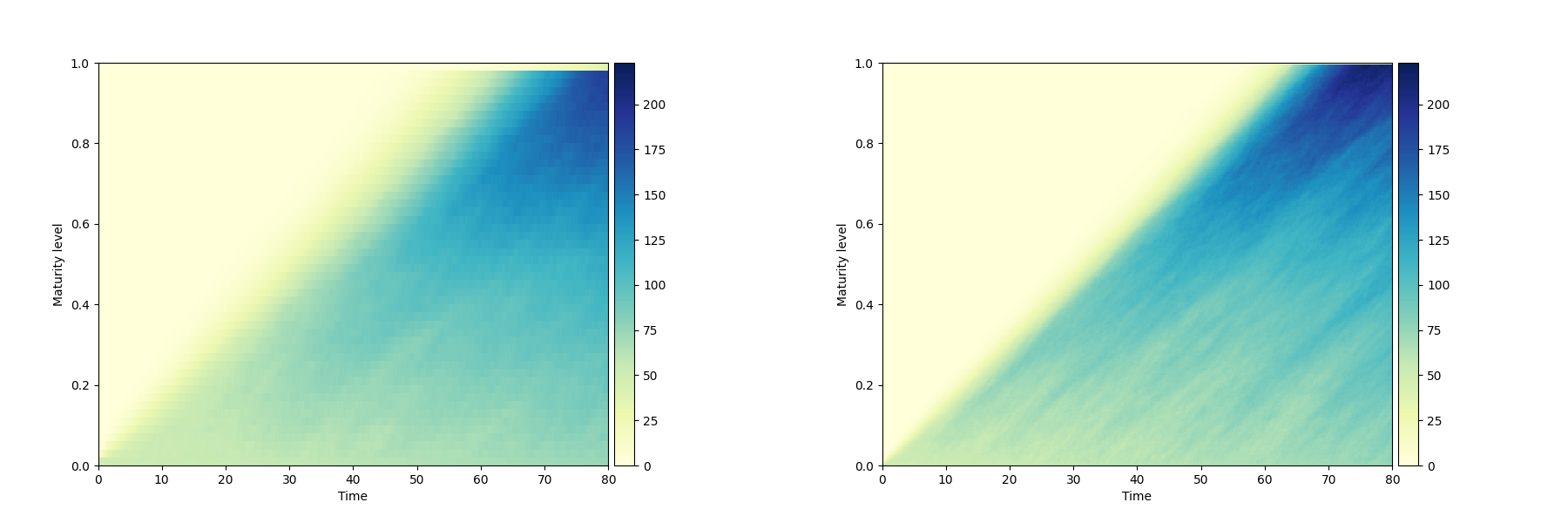}
    \caption{Dynamics of immature cells populations by their maturity level for System \eqref{eq:disc} for $50$ (left) and $200$ (right) compartments.}
     \label{fig:disc}
\end{figure}

\begin{figure}[h]
    \centering
    \includegraphics[width = \textwidth]{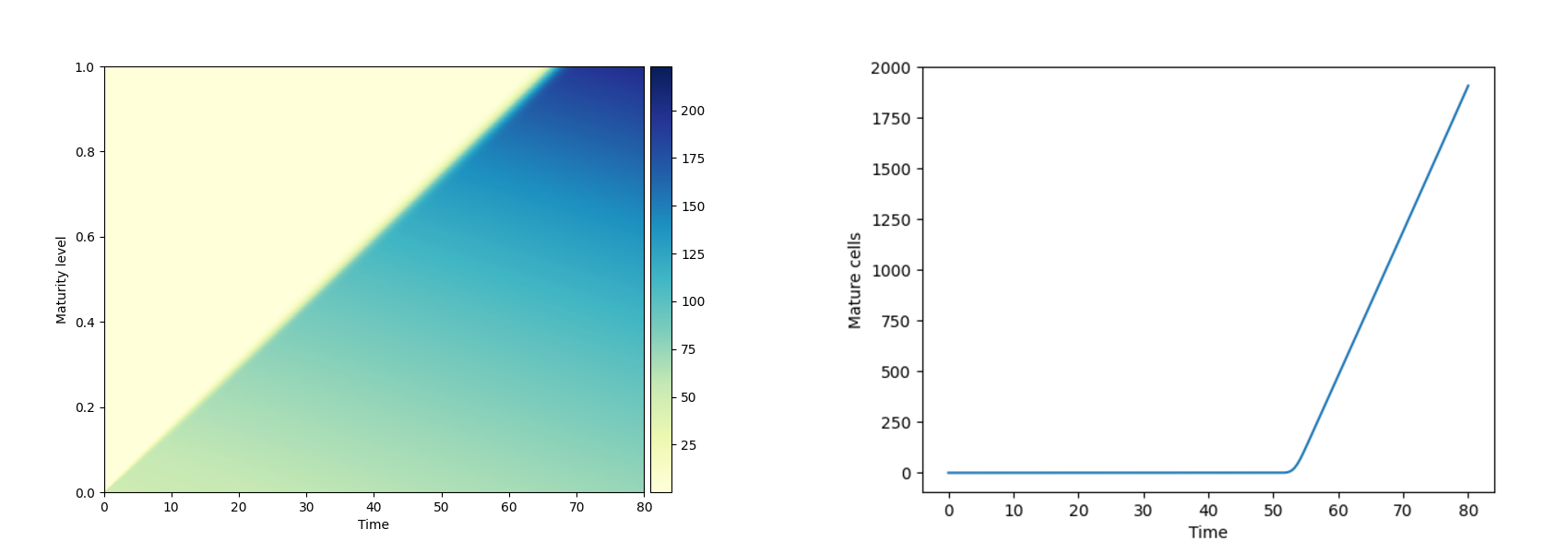}
    \caption{Left: dynamics of immature cells populations by their maturity level for System \eqref{eq:pde}. Right: dynamics of the mature cells population size for System \eqref{eq:pde}.}
    \label{fig:pde}
\end{figure}

\section{Discussion}
In this paper, we presented a stochastic compartmental differential system with continuous time to model hematopoiesis at steady-state with a regulation loop depending only on the number of mature cells. As the goal was to create a continuum in maturity level, we proposed a model where differentiation and division are two separate events, which differs with standard modeling (see, for example, \cite{Bonnet}, \cite{Dingli}, \cite{Knauer}, \cite{MonPere} or \cite{Stiehl}) where cells can only differentiate during a division. Note that in our model, cells differentiate to the immediate next maturation stage which corresponds to the biological context of steady-state hematopoiesis. Our model cannot represent two situations where jumps in differentiation stages have been observed: 1) megakaryocytes can emerge directly from LT-HSCs with a very few number of intermediary steps; 2) in stress erythropoiesis, ST-HSC can directly jump into stress erythroid progenitors \citet{Harandi}.

We proposed a rigorous derivation of a deterministic measure-valued system with two boundary ODEs dynamics and an intermediary PDE (defined in a weak sense) from the rescaled stochastic compartmental system. The scaling parameter captures both the large order of magnitude of the number of HSCs and the number of compartments. The limiting system is a first step to approximate the dynamics of hematopoiesis amplification in a continuum. This model could be extended to include more accurately different biological considerations such as stress or study the lineages separately. 

The convergence result we obtained arises from a slow-fast dynamics, with slow division events, slow death events for the last compartment, and fast differentiation. That explains the structure of the limit where the extremal compartments play a specific role: the first one is only filled by (slow) self-renewal and the last one is only emptied by (slow) death. On the other hand, the fast differentiation dynamics in the intermediary compartments allowed to derive the limiting PDE. The topic of a future work will be to study the associated fluctuations and deduce the speed of convergence. \\ \ \\ \ \\

\noindent\textbf{Statements and Declarations:} The authors have no relevant financial or non-financial interests to disclose. \\
\noindent\textbf{Fundings}: Funded by the European Union (ERC, SINGER, n°101054787). Views and opinions expressed are however those of the author(s) only and do not necessarily reflect those of the European Union or the European Research Council Executive Agency. Neither the European Union nor the granting authority can be held responsible for them. This work has also been supported by ITMO Cancer and by the Chair "Mod\'elisation Math\'ematique et Biodiversit\'e" of Veolia Environnement-Ecole Polytechnique-Museum National d'Histoire Naturelle-Fondation X.\\

\bibliographystyle{plainnat}
\bibliography{sn-bibliography}

@article{Zhang, title = {Hematopoietic Hierarchy – An Updated Roadmap}, journal = {Trends in Cell Biology}, volume = {28}, number = {12}, pages = {976-986}, year = {2018}, note = {Special Issue: Stem Cell Biology}, issn = {0962-8924}, doi = {https://doi.org/10.1016/j.tcb.2018.06.001}, url = {https://www.sciencedirect.com/science/article/pii/S0962892418301028},author = {Zhang, Y. and Gao, S. and Xia, J. and Liu, F.}}

@inbook{Quesenberry, chapter = {The Stem Cell Continuum: A New Model of Stem Cell Regulation}, author = {Quesenberry, P.J. and Colvin, G.A. and Dooner, M.S.}, booktitle = {Handbook of experimental pharmacology}, year = {2006}, publisher = {Springer}, pages = {169--183}, doi = {10.1007/3-540-31265-X_8}}

@article{Bonnet, author = {Bonnet, C. and Gou, P. and Girel, S. and Bansaye, V. and Lacout, C. and Bailly, K. and Schlagetter, M.H. and Lauret, E. and Méléard, S. and Giraudier, S.}, title = {Multistage hematopoietic stem cell regulation in the mouse: A combined biological and mathematical approach}, journal = {iScience}, year = {2021}, volume = {24}, number = {12}, pages = {103399}, issn = {2589-0042}, doi = {https://doi.org/10.1016/j.isci.2021.103399}, url = {https://www.sciencedirect.com/science/article/pii/S2589004221013705},}

@book{Kipnis, author = {Kipnis, C. and Landim, C.}, title = {Scaling limits of interacting particle systems}, publisher = {Springer}, year ={1999}}

@article{Perrut, author = {Perrut, A.}, title = {Hydrodynamic limits for a two-species reaction-diffusion process}, journal = {The annals of applied probability}, year = {2000}, volume = {10}, pages = {163-191}, doi = {10.1214/aoap/1019737668}}

@book{DeMasi, author = {DeMasi, A. and Presutti, E.}, title = {Mathematical methods for hydrodynamic limits}, publisher = {Springer}, year = {1991}}

@article{Doumic, author = {Doumic, M. and Marciniak-Czochra, A. and Perthame, B. and Zubelli, J. P.}, title = {A structured Population Model of Cell Differentiation},journal = {SIAM Journal on Applied Mathematics}, volume = {71}, number = {6}, pages = {1918-1940}, doi = {10.1137/100816584}, URL = {https://doi.org/10.1137/100816584}, year = {2011}}

@article{Aldous, author = {Aldous, D.}, title = {Stopping times and tightness}, journal = {The Annals of Probabability}, year = {1978}, volume = {6}, number ={2}, pages = {335-340}, doi = {10.1214/aop/1176995579}}

@article{Tense, author = {Joffe, A. and Métivier M.}, title = {Weak convergence of sequences of semimartingales with applications to multitype branching processes}, journal = {Advances in Applied Probability}, volume = {18}, issue ={1}, pages ={20-65}, year ={2012}, doi = {10.2307/1427238}}

@article{Wilson, author = {Wilson, N. K. and Kent, D.G. and Buettner, F. and Shehata, M. and Macaulay, I.C. and Calero-Nieto, Fernando J and Sánchez Castillo, Manuel and Oedekoven, C.A. and Diamanti, E. and Schulte, R. and Ponting, C.P. and Voet, T. and Caldas, C. and Stingl, J. and Green, A.R. and Theis, F.J. and Göttgens,  B.}, title = {Combined Single-Cell Functional and Gene Expression Analysis Resolves Heterogeneity within Stem Cell Populations}, journal = {Cell Stem Cell}, year = {2015}, volume = {16,6}, pages = {712-724}, doi = {10.1016/j.stem.2015.04.004}}

@article{Pietras, author = {Pietras, E.M. and Reynaud, D. and Kang, Y.A. and Carlin, D. and Calero-Nieto, F.J. and Leavitt, A.D. and Stuart, J.M. and Gottgens, B. and Passegue, E.}, title = {Functionally distinct subsets of lineage-biased multipotent progenitors control blood production in normal and regenerative conditions}, journal = {Cell Stem Cell}, year = {2015}, volume = {17}, issue = {1}, pages = {35-46}, doi ={10.1016/j.stem.2015.05.003}}

@article{Cheng, author = {Cheng, H. and Zheng, Z. and Cheng, T.}, title = {New paradigms on hematopoietic stem cell differentiation}, journal = {Protein and Cell}, year = {2019}, volume = {11,1}, pages = {34-44}, doi ={10.1007/s13238-019-0633-0}}

@article{Roelly, author = {Roelly-Coppoletta, S.}, title = {A criterion of convergence of measure-valued processes: application to measure branching processes}, journal = {Stochastics}, year = {1986}, volume={17}, pages={43-65}, doi = {10.1080/17442508608833382}}

@book{Kimmel, author = {Kimmel, M. and Axelrod, D.E.}, title = {Branching processes in biology}, publisher = {Springer}, year = {2002}}

@article{Knauer, author = {Knauer, F. and Stiehl, T. and Marciniak-Czochra, A.}, title = {Oscillations in a white blood cell production model with multiple differentiation stages}, journal = {Journal of Mathematical Biology}, year = {2020}, volume = {80}, pages = {575–600}, doi ={10.1007/s00285-019-01432-6}}

@article{MonPere, author = {Mon Père, N.V. and Lenaerts, Tom and dos Santos Pacheco, J.M. and Dingli, D.}, title = {Multistage feedback-driven compartmental dynamics of hematopoiesis}, journal = {iScience}, year = {2021}, volume = {24}, isuue ={4}, doi = {10.1016/j.isci.2021.102326}}

@article{Stiehl, author = {Stiehl, T. and Ho, A.D. and Marciniak-Czochra,  A.}, title = {Mathematical modeling of the impact of cytokine response of acute myeloid leukemia cells on patient prognosis}, journal = {Scientific reports}, year = {2018}, volume = {8}, pages ={1-11}, doi = {10.1038/s41598-018-21115-4}}

@article{Dingli, author = {Dingli, D. and Traulsen, A. and Pacheco, J.M.}, title = {Compartmental Architecture and Dynamics of Hematopoiesis}, journal = {PLoS ONE}, year = {2007}, volume ={2}, issue ={4}, doi ={10.1371/journal.pone.0000345}}

@book{Leveque, author = {Leveque, R.J.}, title = {Finite-volume methods for hyperbolic problems}, publisher = {Cambridge university press}, year = {2002}}

@book{HirschSmale, author = {Hirsch, M. W. and Smale, S.}, title ={Differential equations, dynamical systems and linear algebra}, year ={1974}, publisher = {Academic Press}}

@chapter{Mouhot, author = {Mouhot, C.}, title = {Hyperbolicity: scalar transport equations, wave equations}, url = {https://cmouhot.wordpress.com/wp-content/uploads/1900/10/chapter41.pdf}, year = {2014}}

@article{Kumaravelu, author = {Kumaravelu, P. and Hook, L. and Morrison, A.M. and Ure, J. and Zhao, S. and Zuyev, S. and Ansell, J. and Medvinsky, A.}, title = {Quantitative developmental anatomy of definitive haematopoietic stem cells/long-term repopulating units (HSC/RUs): role of the aorta-gonad-mesonephros (AGM) region and the yolk sac in colonisation of the mouse embryonic liver}, journal = {Development}, year = {2002}, volume ={129}, number ={21}, pages ={4891-4899}, doi ={10.1242/dev.129.21.4891}}

@article{Rybtsov, author = {Rybtsov, S. and Sobiesiak, M. and Taoudi, S. and Souilhol, C. and Senserrich, J. and Liakhovitskaia, A. and Ivanovs, A. and Frampton, J. and Zhao, S. and Medvinsky, A.}, title = {Hierarchical organization and early hematopoietic specification of the developing HSC lineage in the AGM region}, journal = {Journal of Experimental Medicine}, year = {2011}, volume ={208}, number = {6}, pages = {1305-1315}, doi ={10.1084/jem.20102419}}

@article{Tober, author = {Tober, J. and Maijenburg, M.M.W. and Li, Y. and Gao, L. and Hadland, B.K. and Gao, P. and Minoura, K. and Bernstein, I.D. and Tan, K. and Speck, N.A.}, title = {Maturation of hematopoietic stem cells from prehematopoietic stem cells is accompanied by up-regulation of PD-L1}, journal = {Journal of Experimental Medicine}, year = {2018}, volume = {215}, number = {2}, pages = {645-659}, doi ={10.1084/jem.20161594}}

@article{Blount, author ={Blount, D.}, title = {Law of Large Numbers in the Supremum Norm for a Chemical Reaction with Diffusion}, journal ={The Annals of Applied Probability}, volume = {2}, number = {1}, year = {1992}, pages = {131–141}, doi ={10.1214/aoap/1177005774}}

@article{Bahadoran, author = {Bahadoran, C.}, title = {Hydrodynamics and Hydrostatics for a Class of Asymmetric Particle Systems with Open Boundaries}, journal = {Communications in Mathematical Physics}, volume = {310}, pages = {1–24}, year = {2012}, doi ={10.48550/arXiv.math/0612094}}

@book{Athreya, author ={Athreya, K.B. and Ney, P.E.}, title = {Branching Processis}, year = {1972}, edition = {1st}, publisher = {Springer}}

@book{Mode, author = {Mode, C.J.}, title = {Multitype branching processes: Theory and Applications}, year = {1971}, publisher ={Elsevier}}

@article{Smith, author = {Smith, B.R.}, title = {Regulation of hematopoiesis}, journal = {Yale Journal of Biology and Medicine}, year = {1990}, volume = {63}, number = {5}, pages ={371-80}, doi ={}}

@article{Harandi, author = {Harandi, O.F. and Hedge, S. and Wu, D.C. and Mckeone, D. and Paulson, R.F.}, title = {Murine erythroid short-term radioprotection requires a BMP4-dependent, self-renewing population of stress erythroid progenitors}, journal = {The Journal of Clinical Investigation}, year = {2010}, volume = {120}, number = {12}, pages = {4507-4519}, doi ={10.1172/JCI41291}}

\appendix

\section{Appendix}

\begin{appxlem} \label{lem:mart}
    For any $N > 0$, the martingales $(M_j^N(t), 1 \leq j \leq N)_{t \geq 0}$ are square-integrable and their quadratic variations are given by
    \begin{align*}
        \left< M_1^N \right>_t =& \, \frac{1}{N} \, \int_0^t \left[r \left( \frac{1}{N}, \, X_N^N(s) \right) + m\left( \frac{1}{N}, \, X_N^N(s) \right) \right] X_1^N(s) \, ds \\
        \left< M_i^N \right>_t =& \, \int_0^t \left[ r\left( \frac{i}{N}, \, X_N^N(s) \right) + N \, m\left( \frac{i}{N}, \, X_N^N(s) \right) \right] X_i^N(s) \, ds \\
        & + \int_0^t N \, m\left( \frac{i - 1}{N}, \, X_N^N(s) \right) X_{i - 1}^N(s) \, ds, \quad 2 \leq i \leq N - 1, \\
        \left< M_N^N \right>_t =& \, \frac{1}{N} \, \int_0^t \left[m\left( \frac{N - 1}{N}, \, X_N^N(s)\right) X_{N - 1}^N (s) + d \, X_N^N(s)\right] \, ds, \\
        \left< M_1^N, M_2^N \right>_t =& - \int_0^t m \left( \frac{1}{N}, \, X_N^N(s)\right) X_1^N(s) \, ds, \\
        \left< M_i^N, M_{i + 1}^N \right>_t =& - \int_0^t N \, m\left( \frac{i}{N}, \, X_N^N(s)\right) X_i^N(s) \, ds, \quad 2 \leq i \leq N - 2,\\
        \left< M_{N - 1}^N, M_N^N \right>_t =& - \int_0^t m\left( \frac{N - 1}{N}, \, X_N^N(s)\right) X_{N - 1}^N(s)\ ds, \\
        \left< M_i^N, M_j^N \right>_t =& \, 0 \quad \text{ otherwise.}
    \end{align*}
\end{appxlem} \, \\

\begin{proof}
Recall that the compensated measures $(\mathcal{N}_i^j, 1 \leq i \leq N\,, \,j \in \{ r, m, d\})$ are all independent. The martingale $M_1(t)$ is given by
\begin{align*}
    M_1^N(t) =& \frac{1}{N} \, \int_0^t \int_{\mathbb{R}^+} \mathbb{1}_{\theta \leq N \, r(1/N, X_N^N(s^-)) \, X_1^N(s^-)} \, \mathcal{\widetilde{N}}_1^r(ds, \, d\theta) \\&-  \frac{1}{N} \, \int_0^t \int_{\mathbb{R}^+} \mathbb{1}_{\theta \leq N \, m(1/N, X_N^N(s^-)) \, X_1^N(s^-)} \, \mathcal{\widetilde{N}}_1^m(ds, \, d\theta).
\end{align*}
By independence of $\mathcal{\widetilde{N}}_1^r(ds, \, d\theta)$ and $\mathcal{\widetilde{N}}_1^m(ds, \, d\theta)$ we can derive
\begin{align*}
    \left< M_1^N \right>_t =& \frac{1}{N^2} \, \int_0^t \int_{\mathbb{R}^+} \mathbb{1}_{\theta \leq N \, r(1/N, X_N^N(s^-)) \, X_1^N(s^-)} \, d\theta \, ds +  \frac{1}{N^2} \, \int_0^t \int_{\mathbb{R}^+} \mathbb{1}_{\theta \leq N \, m(1/N, X_N^N(s^-)) \, X_1^N(s^-)} \, d\theta \, ds \\
    =& \frac{1}{N} \, \int_0^t   r\left(\frac{1}{N}, X_N^N(s)\right) \, X_1^N(s) \, ds \, +  \frac{1}{N} \, \int_0^t  m\left(\frac{1}{N}, X_N^N(s)\right) \, X_1^N(s) \, d\theta \, ds.
\end{align*}
With similar arguments, for $2 \leq i \leq N - 1$ we obtain
\begin{align*}
    \left< M_i^N \right>_t =& \int_0^t \int_{\mathbb{R}^+} \mathbb{1}_{\theta \leq r(i/N, X_N^N(s^-))\  X_i^N(s^-)} \, d\theta \, ds \\&+ \int_0^t \int_{\mathbb{R}^+} \mathbb{1}_{\theta \leq N \, m((i - 1)/N, X_N^N(s^-)) \, X_{i - 1}^N(s^-)} \, d\theta \, ds + \int_0^t\int_{\mathbb{R}^+} \mathbb{1}_{\theta \leq N \, m(i/N, X_N^N(s^-)) \, X_i^N(s^-)} \, d\theta \, ds \\
    =& \int_0^t r\left(\frac{i}{N}, X_N^N(s)\right)\  X_i^N(s) \, ds + \int_0^t N \, m\left(\frac{i - 1}{N}, X_N^N(s)\right) \, X_{i - 1}^N(s) \, ds\\&+ \int_0^t N m\left(\frac{i}{N}, X_N^N(s)\right) \, X_i^N(s) \, ds,
\end{align*}
and for $M_N(t)$ we deduce
\begin{align*}
    \left< M_N^N \right>_t =& \frac{1}{N^2} \, \int_0^t \int_{\mathbb{R}^+} \mathbb{1}_{\theta \leq N \, m((N - 1)/N, X_N^N(s^-)) \, X_{N - 1}^N (s^-)} \, d\theta \, ds + \frac{1}{N^2} \, \int_0^t \int_{\mathbb{R}^+} \mathbb{1}_{\theta \leq N \, d \, X_N^N(s^-)} \, d\theta \, ds \\
    =& \frac{1}{N} \, \int_0^t \, m\left(\frac{N - 1}{N}, X_N^N(s))\right) \, X_{N - 1}^N (s) \, ds + \frac{1}{N} \, \int_0^t d \, X_N^N(s) \, ds. 
\end{align*}

On the other hand, we have that
\begin{align*}
    \left< M_1^N, M_2^N \right>_t =& -  \frac{1}{N} \, \int_0^t \int_{\mathbb{R}^+} \mathbb{1}_{\theta \leq N \, m(1/N, X_N^N(s^-)) \, X_1^N(s^-)} \, d\theta \, ds \\
    =& -  \int_0^t m \left( \frac{1}{N}, X_N^N(s) \right) \, X_1^N(s) \, ds.
\end{align*}
Similarly, we can find $\left< M_i, M_{i + 1} \right>_t$ for $2 \leq i \leq N - 2$ and $\left< M_{N - 1}, M_N \right>_t$. Lastly, for any other $i, j$, given the independence of their compensated measures, it follows that
\begin{equation*}
    \left< M_i, M_{j} \right>_t = 0.
\end{equation*}

\end{proof}

\begin{appxlem} \label{lem:muk}
   Let us consider $M_F$, the set of finite measures on $\mathbb R$.
   We also consider a measure-valued function $\mu \in \mathbb{D} (\mathbb{R}_+, M_F)$ and a compact domain $D$ of $\mathbb{R}_+ \times \mathbb{R}$ such that 
    \begin{equation*}
        \int_D \mu(s, dx) ds < \infty.
    \end{equation*}
    Then,  for any function $g$ continuous on $D$, we have that
    \begin{equation*}
        \int_D g(s, x) \mu^k(s, dx) ds \rightarrow \int_D g(s, x) \mu(s, dx) ds
    \end{equation*}
    as $k \rightarrow \infty$, where
    \begin{equation}
    \label{def:muk}
        \mu^k(s, dx) = \sum_{i \in \mathbb{Z}} \mu\left( s, \left( \frac{i - 1}{k}, \frac{i}{k} \right] \right) \delta_{\frac{i}{k}}.
    \end{equation}
\end{appxlem}

\begin{proof}
By Tietze Theorem, we extend the function $g$ in a continuous way from the set $D$ to a rectangle $\mathcal{R} = [t_1, t_2] \times [n_1, n_2] \subset \mathbb{R}^2$ which contains $D$ with $n_1, n_2 \in \mathbb{Z}$.

For any $k \in \mathbb{N}$, we divide the set $\mathcal{R}$ in disjoint subsets of the form $\mathcal{R}_i^k = \{(s, x) \in \mathcal{R} : x \in ((i - 1)/k, i/k] \}$ for $i \in \mathbb{Z}$. Notice that $\mathcal{R}_i^k = \emptyset$ for $|i|$ large enough by construction. Then, for any $i \in \mathbb{Z}$ and any $(s, x) \in \mathcal{R}_i^k$ we have that
\begin{equation*}
    \left| g(s, x) - g \left( s, \frac{i}{k} \right) \right| \leq \eta^k,
\end{equation*}
where
\begin{equation} \label{eq:eta}
    \eta^k = \sup_{\substack{(s, x), (s,x') \in \mathcal{R}, \\ |x - x'| \leq 1/k}} |g(x, s) - g(x', s)|.
\end{equation}
By integrating over $\mathcal{R}_i^k$ we obtain
\begin{align*}
    \left| \int_{\mathcal{R}_i^k} g(s, x) \mu(s, dx) \, ds - \int_{\mathcal{R}_i^k} g \left( s, \frac{i}{k} \right) \mu(s, dx) \, ds \right| \leq \eta^k \int_{\mathcal{R}_i^k} \mu(s, dx) \, ds.
\end{align*}
By summing over all $i \in \mathbb{Z}$ in we derive
\begin{equation} \label{eq:boundeta}
    \left| \int_{\mathcal{R}} g(s, x) \, \mu(s, dx)ds - \int_{\mathcal{R}} g \left( s, x \right) \, \mu^k(s, dx) \, ds \right| \leq \eta \int_{D} \mu(s, dx) \, ds,
\end{equation}
since
\begin{align*}
    \sum_{i \in \mathbb{Z}} \int_{\mathcal{R}_i^k} g \left( s, \frac{i}{k} \right) \, \mu(s, dx) \, ds &= \int_{\mathcal{R}} \sum_{i \in \mathbb{Z}}  g \left( s, \frac{i}{k} \right) \mu\left( s, \left( \frac{i - 1}{k}, \frac{i}{k} \right] \right)  ds \\
    &= \int_{\mathcal{R}} g \left( s, x \right) \, \mu^k(s, dx) \, ds.
\end{align*}
Notice that $\eta^k \rightarrow 0$ when $k \rightarrow \infty$ by construction in \eqref{eq:eta} and by assumption $\int_{D} \mu(s, dx)ds < \infty$. Then, using \eqref{eq:boundeta} we prove the Lemma.
\end{proof}

\end{document}